\documentclass[10pt,reqno]{amsart}

\usepackage{color}
\usepackage{mathrsfs}
\usepackage{enumitem}
\usepackage{graphicx}
\usepackage[normalem]{ulem}
\usepackage{amsmath, amssymb, amsfonts, dsfont, amsthm}
\usepackage[utf8]{inputenc}
\DeclareGraphicsExtensions{.pdf,.png,.jpg}
\newtheorem{thm}{Theorem}[section]
\newtheorem{defn}[thm]{Definition}
\newtheorem{corollary}[thm]{Corollary}
\newtheorem{lemma}[thm]{Lemma}

\theoremstyle{remark}
\newtheorem{remark}[thm]{Remark}

\usepackage{comment}

\def\XXint#1#2#3{{\setbox0=\hbox{$#1{#2#3}{\int}$}
\vcenter{\hbox{$#2#3$}}\kern-.5\wd0}}

\newcommand\cbrk{\text{$]$\kern-.15em$]$}}
\newcommand\opar{\text{\,\raise.2ex\hbox{${\scriptstyle
|}$}\kern-.34em$($}}
\newcommand\cpar{\text{$)$\kern-.34em\raise.2ex\hbox{${\scriptstyle |}$}}\,}

\def\<{\langle}
\def\>{\rangle}

\def\E{{\mathbb E}}

\newcommand\bL{\mathbb{L}}
\newcommand\bR{\mathbb{R}}

\newcommand\bH{\mathbb{H}}
\newcommand\bZ{\mathbb{Z}}

\newcommand\bD{\mathbb{D}}
\newcommand\bS{\mathbb{S}}

\newcommand\bN{\mathbb{N}}

\newcommand\cA{\mathcal{A}}
\newcommand\cB{\mathcal{B}}

\newcommand\cD{\mathcal{D}}
\newcommand\cF{\mathcal{F}}
\newcommand\cG{\mathcal{G}}

\newcommand\cP{\mathcal{P}}

\newcommand\cO{\mathcal{O}}

\newcommand\frH{\mathfrak{H}}

\newcommand{\mysection}[1]{\section{#1}
\setcounter{equation}{0}}

\newcommand{\one}{\ensuremath{\mathds 1}}

\newcommand{\prob}{\ensuremath{\mathbb P}}

\newcommand{\pred}{\ensuremath{\mathcal{P}}}

\newcommand{\cont}{\ensuremath{\mathcal{C}}}

\newcommand{\abs}[1]{\ensuremath{\lvert #1 \rvert}}
\newcommand{\Abs}[1]{\ensuremath{\big\lvert  #1 \big\rvert}}

\newcommand{\nnrm}[2]{\ensuremath{\lVert #1 \rVert_{#2}}}
\newcommand{\gnnrm}[2]{\ensuremath{\big\lVert #1 \big\rVert_{#2}}}

\newcommand{\ssgnnrm}[2]{\ensuremath{\bigg\lVert #1 \bigg\rVert_{#2}}}

\newcommand{\nrklam}[1]{(#1)}
\newcommand{\rklam}[1]{\left(#1\right)}
\newcommand{\grklam}[1]{\big(#1\big)}
\newcommand{\sgrklam}[1]{\Big(#1\Big)}
\newcommand{\ssgrklam}[1]{\bigg(#1\bigg)}

\newcommand{\ggklam}[1]{\big\{#1\big\}}
\newcommand{\sggklam}[1]{\Big\{#1\Big\}}
\newcommand{\ssggklam}[1]{\bigg\{#1\bigg\}}

\newcommand{\geklam}[1]{\big [#1 \big ]}
\newcommand{\sgeklam}[1]{\Big [#1 \Big ]}
\newcommand{\ssgeklam}[1]{\bigg [#1 \bigg ]}
\newcommand{\reklam}[1]{\Bigg [#1 \Bigg ]}

\newcommand{\dd}{\,\mathrm d}        
\newcommand{\dx}{\mathrm d x}

\newcommand{\ds}{\mathrm d s}
\newcommand{\dt}{\mathrm d t}
\newcommand{\dy}{\mathrm d y}

\newcommand{\domain}{\ensuremath{\mathcal{O}}}   
\newcommand{\domaina}{\ensuremath{\mathcal{D}}}   
\newcommand{\dist}{\ensuremath{\rho}}   
\newcommand{\distv}{\ensuremath{\dist_\circ}}   

\DeclareFontFamily{U}{matha}{\hyphenchar\font45}
\DeclareFontShape{U}{matha}{m}{n}{
      <5> <6> <7> <8> <9> <10> gen * matha
      <10.95> matha10 <12> <14.4> <17.28> <20.74> <24.88> matha12
      }{}
\DeclareSymbolFont{matha}{U}{matha}{m}{n}
\DeclareFontSubstitution{U}{matha}{m}{n}
\DeclareFontFamily{U}{mathx}{\hyphenchar\font45}
\DeclareFontShape{U}{mathx}{m}{n}{
      <5> <6> <7> <8> <9> <10>
      <10.95> <12> <14.4> <17.28> <20.74> <24.88>
      mathx10
      }{}
\DeclareSymbolFont{mathx}{U}{mathx}{m}{n}
\DeclareFontSubstitution{U}{mathx}{m}{n}
\DeclareMathDelimiter{\vvvert}{0}{matha}{"7E}{mathx}{"17}

\newcommand{\wsob}{\ensuremath{H}}
\newcommand{\bwsob}{\ensuremath{\mathbb{\wsob}}}
\newcommand{\frwsob}{\ensuremath{\mathfrak{\wsob}}}

\includecomment{suggestion}

\includecomment{pcicomment}

\definecolor{felix}{rgb}{0.2,0.2,1.0} 
\definecolor{petru}{rgb}{0.7,0.1,0.1} 
\definecolor{comment}{rgb}{0.1,0.1,0.7} %
\definecolor{alternative}{rgb}{0.0,0.5,0.0} 

\newcommand{\pci}{\color{black}}
\newcommand{\icp}{\color{black}}

\begin{document}

\title[An $L_p$-\pci theory \icp for the stoch. heat eq. on angular domains with mixed weights]
{An $L_p$-\pci theory \icp for the stochastic heat equation on angular domains in $\mathbb{R}^2$ with mixed weights}

\author{Petru A. Cioica-Licht}
\thanks{Very careful proofreading of parts of the manuscript by Felix Lindner is gratefully acknowledged.}
\address{Petru A. Cioica-Licht (n\'e Cioica), Faculty of Mathematics, University of Duisburg-Essen, 45117 Essen, Germany}
\email{petru.cioica-licht@uni-due.de}

\subjclass[2010]{60H15; 35R60, 35K05}

\keywords{Stochastic partial differential equation,
stochastic heat equation,
weighted $L_p$-estimate,
weighted Sobolev regularity,
angular domain,
non-smooth domain,
corner singularity}

\begin{abstract}
We prove a refined $L_p$-estimate ($p\geq 2$) for the stochastic heat equation on angular domains in $\bR^2$ with mixed weights based on both, the distance to the boundary and the distance to the vertex. This way we can capture both causes for singularities of the solution: the incompatibility of noise and boundary condition on the one hand and the influence of boundary singularities (here, the vertex) on the other hand. Higher order $L_p$-Sobolev regularity with mixed weights is also established. 
\end{abstract}

\maketitle

\mysection{Introduction}\label{sec:Introduction}

This paper aligns in the program started in~\cite{CioKimLee+2018,CioKimLee2019} towards a refined $L_p$-theory for stochastic partial differential equations (SPDEs) on non-smooth domains.
As already mentioned therein, there are mainly two effects that influence the regularity of solutions to such equations: 
On the one hand, a certain incompatibility between noise and boundary conditions may produce blow-ups of higher order derivatives near the boundary---even if the boundary was smooth~\cite{Fla1990,Kry1994}. On the other hand, singularities of the boundary of the domain, like, e.g., corners, edges, cusps and any other points where the boundary is not differentiable, may also lead to singuarities of the solution in their vicinity---even if the forcing terms are deterministic and smooth~\cite{Gri1985,JerKen1995};
see also~\cite{Lin2011,Lin2014} for a detailed analysis of the breakdown of Sobolev regularity on polygonal domains in the context of SPDEs.

In a series of papers it has been shown that on smooth domains (at least $\cont^1$) the incompatibility of noise and boundary condition can be captured accurately by means of weighted $L_p$-Sobolev spaces with weights based on the distance to the boundary~\cite{Kim2004, KimKry2004, Kry1994,KryLot1999, KryLot1999b}; see also~\cite{LinVer2018+}, which can be used to reproduce parts of these results by means of~\cite{NeeVerWei2012,NeeVerWei2012b}.
At the same time, the analysis of deterministic equations on domains with corners and edges takes place in weighted $L_p$-Sobolev spaces with weights based on the distance to the boundary singularities, see, for instance,~\cite{Dau1988,KozMazRos1997, KozMazRos2001,MazRos2010} and the bibliographies therein for elliptic equations and~\cite{KozNaz2014,Naz2001,Sol2001,PruSim2007} for parabolic problems.
Thus, in order to capture both effects and their interplay, a system of weights based on a combination of the distance to the boundary and the distance to its singularities suggests itself.

An appropriate system of mixed weights has been introduced in~\cite{CioKimLee2019} for the stochastic heat equation on angular and polygonal domains $\domain\subset\bR^2$. 
In particular, it has been shown that with this system of weights higher order weighted $L_p$-Sobolev regularity ($p\geq 2$) can be established once a solution of low regularity is known to exist.
However, so far, the existence part has been only solved for a very restrictive range of parameters.
In this paper we prove one crucial missing link: 
refined weighted $L_p$-estimate\pci s \icp for the stochastic \pci and deterministic \icp convolution\pci s \icp associated to the stochastic heat equation
\begin{equation}\label{eq:SHE}
\left.
\begin{alignedat}{3}
\dd u 
&= 
(\Delta && u \pci +  f \icp) 
\dd t
+
g^k \dd w^k_t \quad \text{on } \Omega\times(0,T]\times\domaina,	\\
u
&=
0 && \quad \text{on } \Omega\times(0,T]\times\partial\domaina,	\\
u(0)
&=
0 && \quad \text{on } \Omega\times\domaina,
\end{alignedat}
\right\}	
\end{equation}
on angular domains
\begin{equation}\label{domain:angular}
\cD:=\cD_{\kappa_0}:=\ggklam{x=(x_1,x_2)\in \bR^2: x=(r\cos\vartheta,r\sin\vartheta),\; r>0,\;\vartheta\in (0,\kappa_0)}
\end{equation}
of arbitrary angle $0<\kappa_0<2\pi$, see Theorem\pci s\icp~\ref{thm:LpEstimate} \pci and~\ref{thm:LpEstimate:det} \icp below.
\pci These \icp estimate\pci s \icp at hand, \pci existence and higher order \icp $L_p$-regularity for the stochastic heat equation~\eqref{eq:SHE} in suitable weighted $L_p$-Sobolev spaces can be established along the lines of~\cite{CioKimLee2019}, see Section~\ref{sec:MaxReg}, in particular, Theorem~\ref{thm:reg:main}.
Before we start, we present the setting, which is assumed to hold throughout this paper.

\medskip

\noindent\textbf{Setting.} Let $(\Omega,\cF,\prob)$ be a complete probability space, and $\nrklam{\cF_{t}}_{t\geq0}$ be an increasing filtration of $\sigma$-fields $\cF_{t}\subset\cF$, each of which contains all $(\cF,\prob)$-null sets;
$\E$ denotes the expectation operator. 
We assume that on $\Omega$ we are given a family $(w_t^k)_{t\geq0}$, $k\in\bN$, of independent one-dimensional Wiener processes relative to $\nrklam{\cF_{t}}_{t\geq0}$. We fix $T\in(0,\infty)$ and denote by $\cP_T$ the predictable $\sigma$-field on $\Omega_T:=\Omega\times (0,T]$ generated by $\nrklam{\cF_{t}}_{t\geq0}$; $\prob_T:=\prob\otimes\dt$. Moreover, we fix an arbitrary angle $0<\kappa_0<2\pi$ and consider the 
the stochastic heat equation~\eqref{eq:SHE}
on the angular domain
$\domaina=\domaina_{\kappa_0}$.
Throughout, $C$ is used to denote a positive finite constant and the notation $C(a_1,\ldots,a_n)$ or $C_{a_1,\ldots,a_n}$ means that $C$ only depends on the parameters $a_1,\ldots,a_n$. In general, constants may differ at any appearance.

\mysection{\pci $L_p$-\icp estimates}\label{sec:LpEstimateStoch}

In this section we state and prove 
\pci refined \icp weighted $L_p$-estimate\pci s \icp for the stochastic convolution
\begin{equation}\label{eq:stochconv}
(t,x)\mapsto\sum_{k=1}^\infty\int_0^t\int_\domaina G(t-s,x,y)g^k(s,y)\dd y\dd w^k_s,\qquad (t,x)\in(0,T]\times\domaina\pci,\icp
\end{equation}
\pci and for the deterministic convolution
\begin{equation}\label{eq:detconv}
(t,x)\mapsto\int_0^t\int_\domaina G(t-s,x,y)f(s,y)\dd y\dd s,\qquad (t,x)\in(0,T]\times\domaina.
\end{equation}\icp
Here, $G(t,x,y)=G_{\kappa_0}(t,x,y)$ is the Green function for the heat equation on $\cD=\cD_{\kappa_0}$ with zero Dirichlet boundary condition, defined for every $y\in\cD$ as the solution (in the sense of distributions) to the problem
\begin{align*}
&\frac{\partial G(t,x,y)}{\partial t}-\Delta_x G(t,x,y)=\delta_{(0,y)}(t,x)\quad\text{in }\bR\times\cD,\\
&G(t,x,y)=0\quad\text{for }t\in\bR,\,x\in\partial\cD\setminus\{0\},\qquad G(t,x,y)=0\quad\text{for }t<0,
\end{align*}
cf., e.g., \cite[Lemma~3.7]{KozNaz2014}. 
Our weights are based on the distance $\dist:=\mathrm{dist}(\cdot,\domaina)$ to the boundary \emph{and} on the distance $\distv:=\mathrm{dist}(\cdot,\{0\})=\abs{\cdot}$ to the vertex of $\domaina$.
For $1<p<\infty$, $\Theta\in\bR$, and $\theta\in\bR$, we write 
\[
L_{p,\Theta,\theta}(\domaina)
:=
L_p(\domaina,\cB(\domaina),\distv^{\theta-2}\grklam{\tfrac{\dist}{\distv}}^{\Theta-2}\dx;\bR)
\text{ and }
L_{p,\Theta,\theta}(\domaina;\ell_2)
:=
L_p(\domaina,\cB(\domaina),\distv^{\theta-2}\grklam{\tfrac{\dist}{\distv}}^{\Theta-2}\dx;\ell_2)
\]
as well as
\[
\bL_{p,\Theta,\theta}(\domaina,T):=
L_p(\Omega_T,\pred_T,\prob_T;L_{p,\Theta,\theta}(\domaina))
\text{ and }
\bL_{p,\Theta,\theta}(\domaina,T;\ell_2):=
L_p(\Omega_T,\pred_T,\prob_T;L_{p,\Theta,\theta}(\domaina;\ell_2)).
\]

Our main $L_p$-estimate for the \pci stochastic convolution~\eqref{eq:stochconv} \icp reads as follows.

\begin{thm}\label{thm:LpEstimate}
Let \pci $2\leq p<\infty$ \icp and let
\begin{equation}\label{eq:range:thetaTheta}
p\ssgrklam{1-\frac{\pi}{\kappa_0}}
<\theta<
p\ssgrklam{1+\frac{\pi}{\kappa_0}}
\qquad
\text{and}
\qquad
1<\Theta<p+1.
\end{equation}
Then there exists a finite constant $C$, independent of $g$ and $T$, such that
\begin{align*}
\ssgnnrm{t\mapsto\sum_{k=1}^\infty\int_0^t\int_\domaina G(t-s,\cdot,y)g^k(s,y)\dd y\dd w^k_s}{\bL_{p,\Theta-p,\theta-p}(\domaina,T)}
\leq C\,
\nnrm{g}{\bL_{p,\Theta,\theta}(\domaina,T;\ell_2)}.
\end{align*}
\end{thm}

The key ingredient for the proof of this assertion is the following Green function estimate, recently proven in~\cite{KimLeeSeo2020+}, see, in particular, Theorem~2.6 and Remark~2.2 therein as well as~\cite[Section~2]{KozNaz2014}, in particular, Theorem~2.4(3) therein. 
Moreover, we need some subtle estimates, which we prove in Section~\ref{sec:MainAuxiliaryEstimate}.
We use the notation
\begin{equation}\label{eq:RxJx}
R_{x,c}:=\frac{\abs{x}}{\sqrt{c}+\abs{x}}
\quad\text{and}\quad
J_{x,c}:=\frac{\dist(x)}{\sqrt{c}+\dist(x)}
\end{equation}
for arbitrary $x\in\bR^d$ and $c>0$.

\begin{thm}[\cite{KimLeeSeo2020+}]\label{thm:GreenEstimate}
For all $0<\lambda_1,\lambda_2<\frac{\pi}{\kappa_0}$, there exist finite constants $C,\sigma>0$ depending only on $\kappa_0$, $\lambda_1$, and $\lambda_2$, such that
\[
\abs{G(t-s,x,y)}
\leq 
C \,
R_{x,t-s}^{\lambda_1-1}\,
R_{y,t-s}^{\lambda_2-1}\,
J_{x,t-s}\,
J_{y,t-s}\,
\frac{1}{t-s}e^{-\sigma\frac{\abs{x-y}^2}{t-s}}.
\]
\end{thm}

\begin{proof}[Proof of Theorem~\ref{thm:LpEstimate}]
For $x\in\domaina$ let $\mathrm{w}(x):=\tfrac{\dist(x)}{\dist_\circ(x)}$.
We will show that for arbitrary $0<\lambda_1,\lambda_2<\tfrac{\pi}{\kappa_0}$, the stochastic integral operator defined by
\[
(\cG_\textup{s}h)(t,x)
:=
\sum_{k=1}^\infty\int_0^t\int_\domaina G(t-s,x,y)\frac{\abs{x}^{\mu-2}}{\abs{x}^{\mu-1}}\frac{\mathrm{w}(x)^{\tilde{\mu}-2}}{\mathrm{w}(y)^{\tilde{\mu}-1}} h^k(s,y)\dd y\dd w^k_s,
\]
is a bounded operator from $L_{p,2,2}(\domaina,T;\ell_2)$ to $L_{p,2,2}(\domaina,T)$, provided
\begin{equation}\label{eq:range:mutmu}
\frac{2}{p'}-\lambda_1<\mu<\frac{2}{p'}+\lambda_2
\qquad\text{and}\qquad
\frac{1}{p'}<\tilde{\mu}<\frac{1}{p'}+1,
\end{equation} 
where $p':=p/(p-1)$ is the dual exponent of $p$.
This is equivalent to proving the assertion. 

For the rest of the proof we fix $\mu$ and $\tilde{\mu}$ satisfying~\eqref{eq:range:mutmu} and choose $\gamma_1,\tilde{\gamma}_1,\gamma_2,\tilde{\gamma}_2\in\bR$, such that the conditions
\begin{equation}\label{eq:range:g1tg1}
0<\tilde{\gamma}_2<2+\frac{1}{p'}-\tilde{\mu} 
\quad\text{and}\quad
0<\gamma_2+\tilde{\gamma}_2<\lambda_2-\mu+1+\frac{2}{p'}
\end{equation}
as well as
\begin{equation}\label{eq:range:g2tg2}
\frac{1}{p}<\tilde{\gamma}_1<\tilde{\mu}-\frac{1}{p'}+\frac{1}{p}
\quad\text{and}\quad
0<\gamma_1+\tilde{\gamma}_1<\lambda_1+\mu-\frac{2}{p'}
\end{equation}
are satisfied.
Note that such $\gamma_1,\tilde{\gamma}_1,\gamma_2,\tilde{\gamma}_2\in\bR$ exist since we assume~\eqref{eq:range:mutmu} to hold.

We first note that, by the Burkholder-Davis-Gundy inequality and the triangle inequality for $\ell_2$-valued integrals, we obtain for all $0<t<\infty$ and all $x\in\domaina$ that 
\begin{align*}
\E
\Abs{\mathcal{G}_\textup{s}h(t,x)}^p
&\leq C\, 
\E\reklam{\ssgrklam{\int_0^t \sum_{k=1}^\infty \ssgrklam{
\int_\cD G(t-s, x,y) \frac{\abs{x}^{\mu-2}}{\abs{x}^{\mu-1}}\frac{\mathrm{w}(x)^{\tilde{\mu}-2}}{\mathrm{w}(y)^{\tilde{\mu}-1}}\,h^k(s,y)\,\dy}^2\ds}^{p/2}}\\
&\leq C\,
\E\reklam{\ssgrklam{ \int^t_0 \ssgrklam{\int_{\cD}  \Abs{G(t-s, x,y)}\frac{\abs{x}^{\mu-2}}{\abs{x}^{\mu-1}}\frac{\mathrm{w}(x)^{\tilde{\mu}-2}}{\mathrm{w}(y)^{\tilde{\mu}-1}}\,\abs{h(s,y)}_{\ell_2} \,\dy}^2 \ds}^{p/2}},
\end{align*}
where $C$ depends only on $p$. The first inequality can be justified, e.g., by considering $\cG_\textup{s}h(t,x)$ as a real-valued stochastic integral w.r.t.\ a cylindrical Wiener process on $\ell_2$ and applying the Burkholder-Davis-Gundy inequality as stated in \cite[Theorem~4.36]{DaPZab2014}; if $p=2$ it is enough to apply It\^{o}'s isometry. 
Using the Green function estimate from Theorem~\ref{thm:GreenEstimate}
together with Hölder's inequality we obtain 
\begin{align*}
\E\abs{\grklam{\mathcal{G}_\textup{s}h}(t,x)}^p
&\leq C_p\,
\E\ssgeklam{
\ssgrklam{ 
\int_0^t 
\ssgrklam{
\int_{\domaina} 
R_{x,t-s}^{\lambda_1-1}R_{y,t-s}^{\lambda_2-1}J_{x,t-s}J_{y,t-s}\\
&\qquad\qquad\qquad\qquad\qquad\qquad\times\frac{\abs{x}^{\mu-2}}{\abs{y}^{\mu-1}}
\frac{\mathrm{w}(x)^{\tilde{\mu}-2}}{\mathrm{w}(y)^{\tilde{\mu}-1}}
\abs{h(s,y)}_{\ell_2}
\cdot
\frac{e^{-\sigma\frac{\abs{x-y}^2}{t-s}}}{t-s} \,\dy
}^{2}\,\ds}^{p/2}}\\
&\leq C_p\,
\E\ssgeklam{
\ssgrklam{ 
\int_0^t 
\ssgrklam{
\int_{\domaina} 
R_{x,t-s}^{\gamma_1 p }R_{y,t-s}^{\gamma_2p}
J_{x,t-s}^{\tilde{\gamma}_1 p }J_{y,t-s}^{\tilde{\gamma}_2p}
\abs{h(s,y)}_{\ell_2}^p
\cdot
\frac{e^{-\sigma\frac{\abs{x-y}^2}{t-s}}}{t-s} \,\dy
}^{2/p}\,\\
&\qquad\times
\ssgrklam{
\int_{\domaina} 
R_{x,t-s}^{(\lambda_1 -1-\gamma_1)p'}R_{y,t-s}^{(\lambda_2-1-\gamma_2)p'}
J_{x,t-s}^{(1-\tilde{\gamma}_1) p'}J_{y,t-s}^{(1-\tilde{\gamma}_2)p'}\\
&\qquad\qquad\qquad\qquad\qquad\qquad\times
\frac{\abs{x}^{(\mu-2)p'}}{\abs{y}^{(\mu-1)p'}}
\frac{\mathrm{w}(x)^{(\tilde{\mu}-2)p'}}{\mathrm{w}(y)^{(\tilde{\mu}-1)p'}}
\cdot
\frac{e^{-\sigma\frac{\abs{x-y}^2}{t-s}}}{t-s} \,\dy
}^{2/p'}\,\ds}^{p/2}
}\\
&=: C_p\,
\E\ssgeklam{\ssgrklam{\int_0^t I_{1,p}^{2/p}(x,s,t)\times I_{2,p}^{2/p'}(x,s,t)}^{p/2}},
\end{align*}
where
\begin{align*}
I_{1,p}(x,s,t)
&:=
\int_{\domaina} 
R_{x,t-s}^{\gamma_1 p }R_{y,t-s}^{\gamma_2p}
J_{x,t-s}^{\tilde{\gamma}_1 p }J_{y,t-s}^{\tilde{\gamma}_2p}
\abs{h(s,y)}_{\ell_2}^p
\cdot
\frac{e^{-\sigma\frac{\abs{x-y}^2}{t-s}}}{t-s} \,\dy
\\
\intertext{and}
I_{2,p}(x,s,t)
&:=
\int_{\domaina} 
R_{x,t-s}^{(\lambda_1 -1-\gamma_1)p'}R_{y,t-s}^{(\lambda_2-1-\gamma_2)p'}
J_{x,t-s}^{(1-\tilde{\gamma}_1) p'}J_{y,t-s}^{(1-\tilde{\gamma}_2)p'}\\
&\qquad\qquad\qquad\qquad\qquad\qquad\times
\frac{\abs{x}^{(\mu-2)p'}}{\abs{y}^{(\mu-1)p'}}
\frac{\mathrm{w}(x)^{(\tilde{\mu}-2)p'}}{\mathrm{w}(y)^{(\tilde{\mu}-1)p'}}
\cdot
\frac{e^{-\sigma\frac{\abs{x-y}^2}{t-s}}}{t-s} \,\dy.
\end{align*}
Note that, since only $h$ depends on $\omega\in\Omega$, so does $I_{1,p}$, whereas $I_{2,p}$ is purely deterministic. 
By another application of Hölder's inequality to the integral on $(0,t)$ with  H\"older conjugates $p/2\geq 1$ and $p/(p-2)$, we obtain
\begin{equation}
\E
\Abs{\grklam{\mathcal{G}_\textup{s}h}(t,x)}^p
\leq
\E 
\ssgeklam{\int^t_0 I_{1,p}(x,s,t)\,\ds} \cdot 
\gnnrm{I_{2,p}(x,\cdot,t)}{L_{\frac{2(p-1)}{p-2}}(0,t;\bR)}^{p-1},
\end{equation}
where we use the common convention $1/0:=\infty$ for $p=2$. 

\medskip

\noindent\textbf{Step 1.} We first deal with the norm of $I_{2,p}$ and prove that
\begin{equation}\label{eq:I2p:nnrmEstimate}
\gnnrm{I_{2,p}(x,\cdot,t)}{L_{\frac{2(p-1)}{p-2}}(0,t;\bR)}^{p-1}
\leq
C\, \dist(x)^{-2},\quad x\in\domaina.
\end{equation}
We distinguish between the cases $p=2$ and $p>2$, respectively.

\smallskip

\noindent\emph{Case 1. $p=2$.} Since by~\eqref{eq:range:g2tg2}, $2-\tilde{\mu}-\tilde{\gamma}_2>-1/2$ and $\lambda_2-\mu-1-\gamma_1-\tilde{\gamma_1}>-1$, we may apply Lemma~\ref{lem:b1b2a1a2s} for the estimating the inner integral, and obtain
\begin{align*}
&\sup_{0<s<t}\ssggklam{
R_{x,t-s}^{2(\lambda_1-1-\gamma_1)}
\abs{x}^{2(\mu-\tilde{\mu})}
J_{x,t-s}^{2(1-\tilde{\gamma}_1)}
\dist(y)^{2(\tilde{\mu}-2)}\\
&\qquad\qquad\qquad\times
\int_\domaina 
R_{y,t-s}^{2(\lambda_2-1-\gamma_2)}
\abs{y}^{2(\tilde{\mu}-\mu)}
J_{x,t-s}^{2(1-\tilde{\gamma}_2)}
\dist(x)^{2(1-\tilde{\mu})}
\frac{e^{-\sigma\frac{\abs{x-y}^2}{t-s}}}{t-s}
\dd y}
\\
=&\sup_{0<s<t}\ssggklam{
\frac{\abs{x}^{2(\lambda_1-1-\gamma_1+\mu-\tilde{\mu})}}{\grklam{\sqrt{t-s}+\abs{x}}^{2(\lambda_1-1-\gamma_1)}}
\frac{\dist(x)^{2(\tilde{\mu}-1-\tilde{\gamma}_1)}}{\grklam{\sqrt{t-s}+\dist(x)}^{2(1-\tilde{\gamma}_1)}}
\\
&\qquad\qquad\qquad\times
\int_\domaina 
\frac{\abs{y}^{2(\lambda_2-1-\gamma_2+\tilde{\mu}-\mu)}}{\grklam{\sqrt{t-s}+\abs{y}}^{2(\lambda_2-1-\gamma_2)}}
\frac{\dist(y)^{2(2-\tilde{\mu}-\tilde{\gamma}_2)}}{\grklam{\sqrt{t-s}+\dist(y)}^{2(1-\tilde{\gamma}_2)}}
\frac{e^{-\sigma\frac{\abs{x-y}^2}{t-s}}}{t-s}
\dd y}\\
\leq
&\,C \sup_{0<s<t}
\sggklam{
R_{x,t-s}^{2(\lambda_1-1-\gamma_1+\mu-\tilde{\mu})}
J_{x,t-s}^{2(\tilde{\mu}-\tilde{\gamma}_1)}}
\cdot \dist(x)^{-2}.
\end{align*}
Thus, Estimate~\eqref{eq:I2p:nnrmEstimate} is proven if
\[
\sup_{\substack{0<s<t\\ x\in\domaina}}
\sggklam{
R_{x,t-s}^{2(\lambda_1-1-\gamma_1+\mu-\tilde{\mu})}
J_{x,t-s}^{2(\tilde{\mu}-\tilde{\gamma}_1)}}
\leq 
C<\infty.
\]
But this is guaranteed by the fact that $0<J_{x,t-s}\leq R_{x,t-s}\leq 1$ and the restrictions~\eqref{eq:range:g1tg1} on $\gamma_1, \tilde{\gamma}_1$. 

\smallskip

\noindent\emph{Case 2. $p>2$.} In this case $\nnrm{I_{2,p}(x,\cdot,t)}{L_{\frac{2(p-1)}{p-2}}(0,t;\bR)}^{p-1}$ is given by
\begin{align*}
\ssgrklam{&
\int_0^t\ssgrklam{
R_{x,t-s}^{(\lambda_1 -1-\gamma_1)p'}
J_{x,t-s}^{(1-\tilde{\gamma}_1) p'}
\abs{x}^{(\mu-2)p'}
\mathrm{w}(x)^{(\tilde{\mu}-2)p'}\\
&\qquad\times
\int_{\domaina} 
R_{y,t-s}^{(\lambda_2-1-\gamma_2)p'}
J_{y,t-s}^{(1-\tilde{\gamma}_2)p'}
\abs{y}^{(1-\mu)p'}
\mathrm{w}(y)^{(1-\tilde{\mu})p'}
\cdot
\frac{e^{-\sigma\frac{\abs{x-y}^2}{t-s}}}{t-s} \,\dy}^{\frac{2(p-1)}{p-2}}\dd s}^{\frac{p-2}{2}}.
\end{align*}
By Lemma~\ref{lem:b1b2a1a2s}, the inner integral can be estimated by a constant times
\[
\grklam{\abs{x}+\sqrt{t-s}}^{(\tilde{\mu}-\mu)p'}
\grklam{\dist(x)+\sqrt{t-s}}^{(1-\tilde{\mu})p'},
\]
since $2-\tilde{\mu}-\tilde{\gamma}_2>-1/p'$ and $\lambda_2-\mu-1-\gamma_1-\tilde{\gamma_1}>-2/p'$ by~\eqref{eq:range:g2tg2}.
Inserting this and putting terms together yields
\begin{align*}
&\gnnrm{I_{2,p}(x,\cdot,t)}{L_{\frac{2(p-1)}{p-2}}(0,t;\bR)}^{p-1}\\
&\qquad\leq C 
\ssgrklam{
\int_0^t
R_{x,t-s}^{(\lambda_1 -1-\gamma_1+\mu-\tilde{\mu})\frac{2p}{p-2}}
\frac{\dist(x)^{\grklam{\tilde{\mu}-\tilde{\gamma}_1-\frac{p-2}{p}}\frac{2p}{p-2}}\dist(x)^{2-\frac{2p}{p-2}}}{\grklam{\dist(x)+\sqrt{t-s}}^{\grklam{\tilde{\mu}-\tilde{\gamma}_1-\frac{p-2}{p}}\frac{2p}{p-2}+2}}
\dd s}^{\frac{p-2}{2}}\\
&\qquad = C
\ssgrklam{
\int_0^t
R_{x,t-s}^{(\lambda_1 -1-\gamma_1+\mu-\tilde{\mu})\frac{2p}{p-2}}
\frac{\dist(x)^{\grklam{\tilde{\mu}-\tilde{\gamma}_1-\frac{p-2}{p}}\frac{2p}{p-2}}}{\grklam{\dist(x)+\sqrt{t-s}}^{\grklam{\tilde{\mu}-\tilde{\gamma}_1-\frac{p-2}{p}}\frac{2p}{p-2}+2}}
\dd s}^{\frac{p-2}{2}} \dist(x)^{-2}.
\end{align*}
Thus, Estimate~\eqref{eq:I2p:nnrmEstimate} follows if
\[
\sup_{\substack{x\in\domaina\\ t>0}}
\int_0^t
R_{x,t-s}^{(\lambda_1 -1-\gamma_1+\mu-\tilde{\mu})\frac{2p}{p-2}}
\frac{\dist(x)^{\grklam{\tilde{\mu}-\tilde{\gamma}_1-\frac{p-2}{p}}\frac{2p}{p-2}}}{\grklam{\dist(x)+\sqrt{t-s}}^{\grklam{\tilde{\mu}-\tilde{\gamma}_1-\frac{p-2}{p}}\frac{2p}{p-2}+2}}
\dd s
\leq C<\infty.
\]
But this follows from Lemma~\ref{lem:At} if
\[
\tilde{\gamma}_1<\tilde{\mu}-\frac{1}{p'}+\frac{1}{p}
\quad\text{and}\quad
\gamma_1+\tilde{\gamma}_1<\lambda_1+\mu-\frac{2}{p'},
\]
which is satisfied due to~\eqref{eq:range:g1tg1}.

\medskip

\noindent\textbf{Step 2.} We go back to the full estimate

\begin{align*}
\E &\gnnrm{\mathcal{G}_\textup{s}h}{L_p(\domaina\times (0,T))}^p\\
&=
\E
\int_0^\infty
\int_{\domaina}
\Abs{(\mathcal{G}_\textup{s}h)(t,x)}^p
\,\dx
\,\dt\\
&\leq
C\,
\E
\int_0^\infty
\int_{\domaina}
\int_0^t
\int_{\domaina} 
\frac{R_{x,t-s}^{\gamma_1 p }}{\dist(x)^2}R_{y,t-s}^{\gamma_2p}
J_{x,t-s}^{\tilde{\gamma}_1 p }J_{y,t-s}^{\tilde{\gamma}_2p}
\abs{h(s,y)}_{\ell_2}^p
\cdot
\frac{e^{-\sigma\frac{\abs{x-y}^2}{t-s}}}{t-s} \,\dy
\,\ds
\,\dx
\,\dt\\
&=
C\,
\E
\int_0^\infty
\int_{\domaina}
\ssgrklam{
\int_s^\infty
R_{y,t-s}^{\gamma_2p}J_{y,t-s}^{\tilde{\gamma}_2p}
\int_{\domaina} 
\frac{R_{x,t-s}^{\gamma_1 p}J_{x,t-s}^{\tilde{\gamma}_1 p }}{\dist(x)^2}
\cdot
\frac{e^{-\sigma\frac{\abs{x-y}^2}{t-s}}}{t-s} 
\,\dx
\,\dt }
\abs{h(s,y)}_{\ell_2}^p
\,\dy
\,\ds.
\end{align*}
Thus, it suffices to prove that
\begin{align*}
A:=\sup_{\substack{0<s<\infty\\ y\in\domaina}}\int_s^\infty
R_{y,t-s}^{\gamma_2p}J_{y,t-s}^{\tilde{\gamma}_2p}
\int_{\domaina} 
\frac{R_{x,t-s}^{\gamma_1 p}J_{x,t-s}^{\tilde{\gamma}_1 p }}{\dist(x)^2}
\cdot
\frac{e^{-\sigma\frac{\abs{x-y}^2}{t-s}}}{t-s} 
\,\dx
\,\dt
\leq C<\infty.
\end{align*}
For the inner integral we obtain, by another application of Lemma~\ref{lem:b1b2a1a2s},
\begin{align*}
\int_{\domaina} 
\frac{R_{x,t-s}^{\gamma_1 p}J_{x,t-s}^{\tilde{\gamma}_1 p }}{\dist(x)^2}
\cdot
\frac{e^{-\sigma\frac{\abs{x-y}^2}{t-s}}}{t-s} 
\,\dx
&=
\int_\domaina
R_{x,t-s}^{\gamma_1p}
J_{x,t-s}^{\tilde{\gamma}_1p-2}
\grklam{\sqrt{t-s}+\dist(x)}^{-2}
\cdot
\frac{e^{-\sigma\frac{\abs{x-y}^2}{t-s}}}{t-s} 
\,\dx\\
&\leq C\grklam{\sqrt{t-s}+\dist(y)}^{-2},
\end{align*}
since $\gamma_1+\tilde{\gamma}_1>0$ and $\tilde{\gamma}_1>1/p$ by~\eqref{eq:range:g1tg1}.
Thus,
\begin{align*}
A\leq C \sup_{\substack{0<s<\infty\\ y\in\domaina}}\int_s^\infty
R_{y,t-s}^{\gamma_2p}
\frac{\dist(y)^{\tilde{\gamma}_2p}}{\grklam{\sqrt{t-s}+\dist(y)}^{\tilde{\gamma}_2p+2}}
\,\dt
\leq C<\infty,
\end{align*}
since $\gamma_2+\tilde{\gamma}_2>0$ and $\tilde{\gamma}_2>0$ by~\eqref{eq:range:g2tg2}.
\end{proof}

With an analogous technique, we can prove the following $L_p$-estimate for the deterministic convolution~\eqref{eq:detconv}.

\begin{thm}\label{thm:LpEstimate:det}
Let  $1<p<\infty$ and let $\Theta,\theta\in\bR$ satisfy~\eqref{eq:range:thetaTheta}.
Then there exists a finite constant $C$, independent of $g$ and $T$, such that
\begin{align*}
\ssgnnrm{t\mapsto\int_0^t\int_\domaina G(t-s,\cdot,y)f(s,y)\dd y\dd s}{\bL_{p,\Theta-p,\theta-p}(\domaina,T)}\leq C\, \nnrm{f}{\bL_{p,\Theta+p,\theta+p}(\domaina,T)}.
\end{align*}
\end{thm}

\begin{proof}
Since the proof follows the lines of the proof of Theorem~\ref{thm:LpEstimate} very closely, it is left to the reader.
\end{proof}

\mysection{Weighted $L_p$-Sobolev regularity for the stochastic heat equation}\label{sec:MaxReg}

In this section we use the $L_p$-estimates proven in Section~\ref{sec:LpEstimateStoch} to establish existence 
and higher order regularity for the stochastic heat equation~\eqref{eq:SHE}.
We use weighted Sobolev spaces $H^n_{p,\Theta,\theta}(\domaina)$ with mixed weights to measure the spatial regularity. They are defined as follows:
For $n\in\bN_0$, $\Theta,\theta\in\bR$, and $1<p<\infty$, let $\wsob^n_{p,\Theta,\theta}(\domaina)$ be the collection of all (equivalence classes of) locally integrable real-valued functions $u$ on $\domaina$ such that
\[
\nnrm{u}{H^n_{p,\Theta,\theta}(\domaina)}^p
:=
\sum_{\abs{\alpha}\leq n}\int_\domaina \Abs{\dist^{\abs{\alpha}}D^\alpha u}^p\distv^{\theta-2}\sgrklam{\frac{\dist}{\distv}}^{\Theta-2}\dd x
=
\sum_{\abs{\alpha}\leq n} \nnrm{\dist^{\abs{\alpha}}D^\alpha u}{L_{p,\Theta,\theta}(\domaina,T)}^p
<\infty,
\]
where for $\alpha\in\bN_0^2$, $D^\alpha=D^\alpha_x$ is the $\alpha$ partial derivative with respect to $x\in\domaina$ and, as before, $\dist$ denotes the distance to the boundary $\partial\domaina$ while $\distv$ is the distance to the vertex $x_0:=0$ of the domain $\domaina$.
Moreover, we define $H^{-n}_{p,\Theta,\theta}(\domaina):=(H^n_{p',\Theta',\theta'}(\domaina))^*$ to be the dual of  $H^n_{p',\Theta',\theta'}(\domaina)$ with
\[
\frac{1}{p}+\frac{1}{p'}=1,\quad
\frac{\Theta}{p}+\frac{\Theta'}{p'}=2,\quad
\frac{\theta}{p}+\frac{\theta'}{p'}=2.
\]
For $\ell_2$-valued functions and $n\in\bZ$, the spaces $H^n_{p,\Theta,\theta}(\domaina;\ell_2)$ are defined analogously with $\abs{\cdot}$ replaced by $\abs{\cdot}_{\ell_2}$.
Obviously, $(\wsob^n_{p,\Theta,\theta}(\domaina),\nnrm{\cdot}{H^n_{p,\Theta,\theta}(\domaina)})$ and $(\wsob^n_{p,\Theta,\theta}(\domaina;\ell_2),\nnrm{\cdot}{H^n_{p,\Theta,\theta}(\domaina;\ell_2)})$ are separable, reflexive Banach spaces for every $n\in\bZ$. 

%
%

To formulate our conditions on the different parts of the equation we will use the spaces
\[
\bwsob^n_{p,\Theta,\theta}(\domaina,T):=
L_p(\Omega_T,\pred_T,\prob\otimes\dt;\wsob^n_{p,\Theta,\theta}(\domaina))
\]
of $p$-Bochner integrable $\wsob^n_{p,\Theta,\theta}(\domaina)$-valued predictable stochastic processes; $\bwsob^n_{p,\Theta,\theta}(\domaina,T;\ell_2)$ is defined analogously.
Obviously, for $n=0$, 
\[
L_{p,\Theta,\theta}(\domaina)=\wsob^0_{p,\Theta,\theta}(\domaina),
\quad
\bL_{p,\Theta,\theta}(\domaina,T)=\bwsob^0_{p,\Theta,\theta}(\domaina,T),
\quad\text{and}\quad
\bL_{p,\Theta,\theta}(\domaina,T;\ell_2)=\bwsob^0_{p,\Theta,\theta}(\domaina,T;\ell_2).
\]
Our analysis of Equation~\eqref{eq:SHE} will take place within the following spaces.

\begin{defn}\label{def:frH}
For $2\leq p<\infty$, $n\in\bN$, and $\Theta,\theta\in\bR$ we write $\frH^n_{p,\Theta,\theta,0}(\domaina,T)$ for the collection of all $u\in\bwsob^n_{p,\Theta-p,\theta-p}(\domaina,T)$ for which there exist $f\in \bwsob^{n-2}_{p,\Theta+p,\theta+p}(\domaina,T)$ and $g\in \bwsob^{n-1}_{p,\Theta,\theta}(\domaina,T;\ell_2)$, such that
\[
\mathrm{d}u=f\dd t + g^k\dd w^k_t,\quad t\in (0,T],
\]
on $\domaina$ in the sense of distributions with $u(0,\cdot)=0$, that is, for any $\varphi \in
\cont^{\infty}_{0}(\domaina)$, with probability one, the equality
\begin{equation}\label{eq:distribution}
(u(t,\cdot),\varphi)=   \int^{t}_{0}
(f(s,\cdot),\varphi) \dd s + \sum^{\infty}_{k=1} \int^{t}_{0}
(g^k(s,\cdot),\varphi)\,  dw^k_s
\end{equation}
holds for all $t \leq T$. 
In this situation
we also write
$$
\bD u:=f\qquad\text{and}\qquad \bS u :=g
$$
for the deterministic part and the stochastic part, respectively. Moreover, 
\[
\nnrm{u}{\frH^n_{p,\Theta,\theta,0}(\domaina,T)}
:=
\nnrm{u}{\bwsob^n_{p,\Theta-p,\theta-p}(\domaina,T)}
+
\nnrm{\bD u}{\bwsob^{n-2}_{p,\Theta+p,\theta+p}(\domaina,T)}
+
\nnrm{\bS u}{\bwsob^{n-1}_{p,\Theta,\theta}(\domaina,T;\ell_2)},
\quad
u\in \frH^n_{p,\Theta,\theta,0}(\domaina,T).
\]
\end{defn}

\begin{remark}\label{rem:frH}
\begin{enumerate}[leftmargin=*,label=\textup{(\roman*)}] 
\item In the setting of Definition~\ref{def:frH}, $(\frH^n_{p,\Theta,\theta,0}(\domaina,T),\nnrm{\cdot}{\frH^n_{p,\Theta,\theta,0}(\domaina,T)})$ is a Banach space. The latter is not immediately clear but can be proven, for instance, by extending~\cite[Lemma~2.7]{CioKimLee2019} appropriately. The details are presented in Appendix~\ref{sec:frH}.
\item\label{rem:frH:frK} Hardy's inequality and the fact that $\cont_0^\infty(\domaina)$ is densely embedded in $H^1_{p,\Theta,\theta}(\domaina)$ for all $1<p<\infty$ and $\Theta,\theta\in\bR$ (see Lemma~\ref{lem:wSob:test:dense}) can be used to verify that for arbitrary $\theta\in\bR$ and $1<p<\infty$, $H^1_{p,d-p,\theta-p}(\domaina)=\mathring{K}^1_{p,\theta-p}(\domaina)$ with equivalent norms, where $\mathring{K}^1_{p,\theta-p}(\domaina)$ is the closure of $\cont_0^\infty(\domaina)$ with respect to the norm
\[
\nnrm{u}{K^1_{p,\theta-p}(\domaina)}
=
\ssgrklam{\sum_{\abs{\alpha}\leq 1}\int_\domaina \Abs{\distv^{\abs{\alpha}}D^\alpha u}^p\distv^{\theta-p-2}\dd x}^{1/p},
\quad
u\in\cont_0^\infty(\domaina).
\]
 As a consequence, the space $\frwsob^1_{p,d,\theta,0}(\domaina,T)$ coincides (equivalent norms) with the space $\mathcal{K}^1_{p,\theta,0}(\domaina)$ introduced in \cite[Definition~2.2]{CioKimLee2019} and with the space $\mathfrak{K}^1_{p,\theta}(\domain,T)$ introduced in~\cite[Definition~3.4]{CioKimLee+2018}, 
see also \cite[Remark~2.3]{CioKimLee2019}.
\end{enumerate}
\end{remark}

In this article, Equation~\eqref{eq:SHE} is to be understood in the following way.

\begin{defn}
Let $2\leq p<\infty$, $n\in\bN$, and $\Theta,\theta\in\bR$. We say that $u$ is the solution to Eq.~\eqref{eq:SHE} in the class $\frH^n_{p,\Theta,\theta,0}(\domaina,T)$ if $u\in \frH^n_{p,\Theta,\theta,0}(\domaina,T)$ with
\[
\bD u=\Delta u +f
\qquad\text{and}\qquad
\bS u = g.
\]
\end{defn}

Our main result concerning existence and regularity for Equation~\eqref{eq:SHE} reads as follows.

\begin{thm}\label{thm:reg:main}
Let $2\leq p<\infty$, let $n\in\bN_0$, and let $\Theta,\theta\in\bR$ satisfy~\eqref{eq:range:thetaTheta}. 
Furthermore, assume that 
$f\in\bwsob^{(n-1)\lor 0}_{p,\Theta+p,\theta+p}(\domaina,T)$, $g\in\bwsob^n_{p,\Theta,\theta}(\domaina,T;\ell_2)$,
and let
\begin{equation}\label{eq:stochconv:sol}
u_{\bS,g}(t,x)
=
\sum_{k=1}^\infty\int^t_0 \int_{\domaina} G(t-s,x,y)g^k(s,y)\,\mathrm dy \,\mathrm dw^k_s   
\end{equation}
be the stochastic convolution, while
\begin{equation}\label{eq:detconv:sol}
u_{\bD,f}(t,x)
=
\int^t_0 \int_{\domaina} G(t-s,x,y)f(s,y)\dd y \dd s   
\end{equation}
is the deterministic convolution.
Then $u:=u_{\bD,f}+u_{\bS,g}$ is a solution to Eq.~\eqref{eq:SHE} in the class $\frwsob^{n+1}_{p,\Theta,\theta,0}(\domaina,T)$. Moreover,
\begin{equation}\label{eq:a-priori:1}
\nnrm{u}{\frwsob^{n+1}_{p,\Theta,\theta,0}(\domaina,T)}^p
\leq C\,
\grklam{
\nnrm{f}{\bwsob^{(n-1)\lor 0}_{p,\Theta+p,\theta+p}(\domaina,T)}^p
+
\nnrm{g}{\bwsob^n_{p,\Theta,\theta}(\domaina,T;\ell_2)}^p
}
\end{equation}
with a constant $C$ that does not depend on $T$.
\end{thm}

\begin{proof}
\noindent\emph{Step 1.} We first consider `nice' forcing terms by assuming that $f\in\bwsob^{(n-1)\lor 0}_{p,\Theta+p,\theta+p}(\domaina,T)\cap \bL_{p,d,\theta+p}(\domaina,T)$ and  $g\in \bwsob^{n}_{p,\Theta,\theta}(\domaina,T;\ell_2)\cap \bL_{p,d,\theta}(\domaina,T;\ell_2)$.
Then we already know that $u=u_{\bD,f}+u_{\bS,g}$ is a solution to Eq.~\eqref{eq:SHE} in the class $\frwsob^1_{p,d,\theta,0}(\domaina,T)$ due to~\cite[Theorem~3.7]{CioKimLee+2018}, see also Remark~\ref{rem:frH}\ref{rem:frH:frK} above.
Moreover, by~\cite[Theorem~2.5]{CioKimLee2019} it holds that 
\begin{align*}
\nnrm{u}{\bwsob^{n+1}_{p,\Theta-p,\theta-p}(\domaina,T)}^p
\leq C\,
\sgrklam{
\nnrm{u}{\bL_{p,\Theta-p,\theta-p}(\domaina,T)}^p
+
\nnrm{f}{\bwsob^{(n-1)\lor 0}_{p,\Theta+p,\theta+p}(\domaina,T)}^p
+
\nnrm{g}{\bwsob^{n}_{p,\Theta,\theta}(\domaina,T;\ell_2)}^p
}
\end{align*}
with a constant that does not depend on $T$. Thus, by Theorems~\ref{thm:LpEstimate} and~\ref{thm:LpEstimate:det}, Estimate~\eqref{eq:a-priori:1} holds and $u$ is a solution to Eq.~\eqref{eq:SHE} in the class $\frwsob^{n+1}_{p,\Theta,\theta,0}(\domaina,T)$. 

\noindent\emph{Step 2.} Let us now assume that $f\in\bwsob^{(n-1)\lor 0}_{p,\Theta+p,\theta+p}(\domaina,T)$ and $g\in \bwsob^{n}_{p,\Theta,\theta}(\domaina,T;\ell_2)$. Then there exist sequences $(f_{j})_{j\in\bN}\subseteq \bwsob^{(n-1)\lor 0}_{p,\Theta+p,\theta+p}(\domaina,T)\cap \bL_{p,d,\theta+p}(\domaina,T)$ and $(g_j)_{j\in\bN}\subseteq \bwsob^n_{p,\Theta,\theta}(\domaina,T;\ell_2)\cap \bL_{p,d,\theta}(\domaina,T;\ell_2)$, such that
\[
\lim_{j\to\infty}
\sgrklam{
\nnrm{f_j-f}{\bwsob^{(n-1)\lor 0}_{p,\Theta+p,\theta+p}(\domaina,T)}
+
\nnrm{g_j-g}{\bwsob^n_{p,\Theta,\theta}(\domaina,T;\ell_2)}}=0.
\]
By Step~1,  for all $j\in\bN$ it holds that $u_j:=u_{\bD,f_j}+u_{\bS,g_j}\in\frwsob^{n+1}_{p,\Theta,\theta}(\domaina,T)$ and 
\[
\dd u_j=(\Delta u_j + f_j)\dd t + g^k_j\dd w^k_t, \quad t\in (0,T],
\]
on $\domaina$ in the sense of distributions with $u_j(0,\cdot)=0$.
Moreover, 
\[
\nnrm{u_j}{\frwsob^{n+1}_{p,\Theta,\theta,0}(\domaina,T)}^p
\leq C\,
\grklam{
\nnrm{f_j}{\bwsob^{(n-1)\lor 0}_{p,\Theta+p,\theta+p}(\domaina,T)}^p
+
\nnrm{g_j}{\bwsob^n_{p,\Theta,\theta}(\domaina,T;\ell_2)}^p},
\quad j\in\bN,
\] 
and
\[
\nnrm{u_j-u_i}{\frwsob^{n+1}_{p,\Theta,\theta,0}(\domaina,T)}
\leq C\,
\grklam{
\nnrm{f_j-f_i}{\bwsob^{(n-1)\lor 0}_{p,\Theta+p,\theta+p}(\domaina,T)}^p
+
\nnrm{g_j-g_i}{\bwsob^n_{p,\Theta,\theta}(\domaina,T;\ell_2)}},
\qquad i,j\in\bN,
\]
with a constant $C$ independent of $T$. 
Thus, due to the completeness of $\frwsob^{n+1}_{p,\Theta,\theta,0}(\cD,T)$, there exists a limit $\tilde{u}\in\frwsob^{n+1}_{p,\Theta,\theta,0}(\domaina,T)$ of $(u_j)_{j\in\bN}$ in $\frwsob^{n+1}_{p,\Theta,\theta,0}(\domaina,T)$.
Its deterministic part is given by $\bD \tilde{u}=\Delta \tilde{u}$ and its stochastic part is $\bS \tilde{u}=g$. 
In other words, $\tilde{u}$ is a solution to Eq.~\eqref{eq:SHE} in the class $\frwsob^{n+1}_{p,\Theta,\theta,0}(\domaina,T)$.
It remains to prove that $\tilde{u}$ is a version of $u=u_{\bD,f}+u_{\bS,g}$. To see this note that for all $j\in\bN$,
\begin{align*}
\nnrm{\tilde{u}-u}{\bL_{p,\Theta-p,\theta-p}(\domaina,T)}
\leq 
\nnrm{\tilde{u}-u_j}{\bL_{p,\Theta-p,\theta-p}(\domaina,T)}
+
\nnrm{u_j-u}{\bL_{p,\Theta-p,\theta-p}(\domaina,T)}.
\end{align*}
The first norm on the right hand side goes to zero for $j\to\infty$ since $(u_j)_{j\in\bN}$ converges to $\tilde{u}$ in $\frwsob^1_{p,\Theta,\theta,0}(\domaina,T)$ and therefore in $\bL_{p,\Theta-p,\theta-p}(\domaina,T)$ for $j\to\infty$. The second term converges to zero due to the estimates from Theorem~\ref{thm:LpEstimate} and Theorem~\ref{thm:LpEstimate:det}. Thus, $\tilde{u}=u$ in $\bL_{p,\Theta-p,\theta-p}(\domaina,T)$ and~\eqref{eq:a-priori:1} holds.
\end{proof}

\mysection{Auxiliary estimates}\label{sec:MainAuxiliaryEstimate}

The proof of Theorem~\ref{thm:LpEstimate} is based on the two auxiliary results presented in Lemma~\ref{lem:b1b2a1a2s} and Lemma~\ref{lem:At} below. As before, we fix an arbitrary angle $0<\kappa_0<2\pi$ and write $\domaina=\domaina_{\kappa_0}$ for the dihedral angle defined in~\eqref{domain:angular}. Recall the definition of $R_{x,c}$ and $J_{x,c}$ for $x\in\domaina$ and $c>0$ from~\eqref{eq:RxJx}.

\begin{lemma}\label{lem:b1b2a1a2s}
For all $\sigma>0$, all $\alpha_1,\alpha_2\in\bR$ and all $\beta_1,\beta_2\in\bR$ such that $\beta_2>-1$ and $\beta_1+\beta_2>-2$, there exists a constant $C=C(\sigma,\beta_1,\beta_2,\alpha_1,\alpha_2)<\infty$ such that
\begin{equation}\label{eq:b1b2a1a2s}
\sup_{x\in\domaina}
\int_\domaina
\ssgrklam{\frac{\abs{y}}{1+\abs{y}}}^{\beta_1}
\ssgrklam{\frac{\dist(y)}{1+\dist(y)}}^{\beta_2}
\ssgrklam{\frac{1+\abs{y}}{1+\abs{x}}}^{\alpha_1}
\ssgrklam{\frac{1+\dist(y)}{1+\dist(x)}}^{\alpha_2}
e^{-\sigma\abs{x-y}^2}\dd y
\leq C.
\end{equation}
In particular, for all $s<t$ and all $x\in\domaina$ it holds that
\begin{align*}
\int_\domaina
R_{y,t-s}^{\beta_1}
J_{y,t-s}^{\beta_2}
\grklam{\sqrt{t-s}+\abs{y}}^{\alpha_1}
&\grklam{\sqrt{t-s}+\dist(y)}^{\alpha_2}
\frac{e^{-\sigma\frac{\abs{x-y}^2}{t-s}}}{t-s}\dd y\\
&\qquad\leq C\grklam{\sqrt{t-s}+\abs{x}}^{\alpha_1}
\grklam{\sqrt{t-s}+\dist(x)}^{\alpha_2}
\end{align*}
with $C\in (0,\infty)$ independent of $x,t,s$.

\end{lemma}

\begin{proof}
The first estimate follows from the special cases treated in Lemma~\ref{lem:b1b2s}, Lemma~\ref{lem:b1a1s}, and Lemma~\ref{lem:a2s} below by a simple application of Hölder's inequality.
And the second estimate follows from the first one, since $\abs{ax}=a\abs{x}$ and $\dist(ax)=a\dist(x)$ for all $a>0$ and all $x\in\domaina$, so that
\begin{align*}
&\int_\domaina
\ssgrklam{\frac{\abs{y}}{\sqrt{t-s}+\abs{y}}}^{\beta_1}
\ssgrklam{\frac{\dist(y)}{\sqrt{t-s}+\dist(y)}}^{\beta_2}\\
&\qquad\qquad\qquad\qquad\times
\ssgrklam{\frac{\sqrt{t-s}+\abs{y}}{\sqrt{t-s}+\abs{x}}}^{\alpha_1}
\ssgrklam{\frac{\sqrt{t-s}+\dist(y)}{\sqrt{t-s}+\dist(x)}}^{\alpha_2}
\frac{e^{-\sigma\frac{\abs{x-y}^2}{t-s}}}{t-s}\dd y\\
= &\int_\domaina
\ssgrklam{\frac{\abs{\tfrac{1}{\sqrt{t-s}}y}}{1+\abs{\tfrac{1}{\sqrt{t-s}}y}}}^{\beta_1}
\ssgrklam{\frac{\dist(\tfrac{1}{\sqrt{t-s}}y)}{1+\dist(\tfrac{1}{\sqrt{t-s}}y)}}^{\beta_2}\\
&\qquad\qquad\qquad\qquad\times
\ssgrklam{\frac{1+\abs{\tfrac{1}{\sqrt{t-s}}y}}{1+\abs{\tfrac{1}{\sqrt{t-s}}x}}}^{\alpha_1}
\ssgrklam{\frac{1+\dist(\tfrac{1}{\sqrt{t-s}}y)}{1+\dist(\tfrac{1}{\sqrt{t-s}}x)}}^{\alpha_2}
\frac{e^{-\sigma\Abs{\tfrac{1}{\sqrt{t-s}}(x-y)}^2}}{t-s}\dd y \\
= & \int_\domaina
\ssgrklam{\frac{\abs{z}}{1+\abs{z}}}^{\beta_1}
\ssgrklam{\frac{\dist(z)}{1+\dist(z)}}^{\beta_2}\\
&\qquad\qquad\qquad\qquad\times
\ssgrklam{\frac{1+\abs{z}}{1+\abs{\tfrac{1}{\sqrt{t-s}}x}}}^{\alpha_1}
\ssgrklam{\frac{1+\dist(z)}{1+\dist(\tfrac{1}{\sqrt{t-s}}x)}}^{\alpha_2}
e^{-\sigma\Abs{\tfrac{1}{\sqrt{t-s}}x-z}^2}\dd z\\
\leq &\, C(\sigma,\beta_1,\beta_2,\alpha_1,\alpha_2)<\infty.\qedhere
\end{align*}
\end{proof}

We first prove~\eqref{eq:b1b2a1a2s} with $\alpha_1=\alpha_2=0$.

\begin{lemma}\label{lem:b1b2s}
For all $\sigma>0$ and all $\beta_1,\beta_2\in\bR$ such that $\beta_2>-1$ and $\beta_1+\beta_2>-2$, there exists a constant $C=C(\sigma,\beta_1,\beta_2)<\infty$ such that
\begin{align*}
\sup_{x\in\domaina}
\int_\domaina
\ssgrklam{\frac{\abs{y}}{1+\abs{y}}}^{\beta_1}
\ssgrklam{\frac{\dist(y)}{1+\dist(y)}}^{\beta_2}
e^{-\sigma\abs{x-y}^2}\dd y
\leq C.
\end{align*}
\end{lemma}

\begin{proof}
We first prove that, under the given assumptions on $\beta_1,\beta_2\in\bR$ it holds that
\begin{equation}\label{eq:intb1b2:1}
\sup_{x\in\domaina}
\int_{\domaina}\one_{\{\rho\leq 1\}}(y)\ssgrklam{\frac{\abs{y}}{1+\abs{y}}}^{\beta_1}
\ssgrklam{\frac{\dist(y)}{1+\dist(y)}}^{\beta_2}e^{-\sigma\abs{x-y}^2}\dd y<\infty.
\end{equation}
To see this, first note that for arbitrary $a\geq 0$ it holds that
\[
\frac{a}{1+a}\leq \grklam{a\land 1}\leq 2\frac{a}{1+a}.
\]
Therefore,
\begin{align*}
\int_{\domaina}&\one_{\{\rho\leq 1\}}(y)\ssgrklam{\frac{\abs{y}}{1+\abs{y}}}^{\beta_1}
\ssgrklam{\frac{\dist(y)}{1+\dist(y)}}^{\beta_2}
e^{-\sigma\abs{x-y}^2}\dd y\\
&\leq C_{\beta_1,\beta_2}\ssgrklam{
\int_{\domaina}\one_{\{\rho_\circ\leq 1\}}(y)\abs{y}^{\beta_1}
\rho(y)^{\beta_2}\dd y
+
\int_{\domaina}\one_{\{\rho\leq 1,\rho_\circ\geq 1\}}(y)
\rho(y)^{\beta_2}e^{-\sigma\abs{x-y}^2}\dd y}\\
&=:C_{\beta_1,\beta_2}\grklam{I + I\!I},
\end{align*}
where
\begin{align*}
I
&=
\int_0^1\int_0^{\frac{\pi}{2}\land\frac{\kappa_0}{2}}r^{\beta_1} r^{\beta_2}\sin(\vartheta)^{\beta_2} r\dd r\dd \vartheta\\
&\qquad+
\int_0^1\int_{\frac{\pi}{2}\land\frac{\kappa_0}{2}}^{\kappa_0-\rklam{\frac{\pi}{2}\land\frac{\kappa_0}{2}}}r^{\beta_1} r^{\beta_2} r\dd r\dd \vartheta
+
\int_0^1\int_{\kappa_0-\rklam{\frac{\pi}{2}\land\frac{\kappa_0}{2}}}^{\kappa_0} r^{\beta_1} r^{\beta_2}\sin(\kappa_0-\vartheta)^{\beta_2} r\dd r\dd \vartheta\\
&\leq C_{\beta_2}
\int_0^1 r^{\beta_1+\beta_2+1}\dd r
\int_0^{\pi/2}\sin(\vartheta)^{\beta_2}\dd\vartheta 
\leq C_{\beta_1,\beta_2}<\infty,
\end{align*}
since $\beta_1+\beta_2>-2$, and
\begin{align*}
I\!I\leq 2 \int_0^\infty\int_0^1 y_2^{\beta_2}\dd y_2\cdot e^{-\sigma\abs{x_1-y_1}^2} \dd y_1\leq C_{\beta_2,\sigma}<\infty,
\end{align*}
since $\beta_2>-1$ by assumption. Thus,~\eqref{eq:intb1b2:1} holds. Moreover,
\begin{align*}
\int_{\domaina}\one_{\{\rho\geq 1\}}(y)&\ssgrklam{\frac{\abs{y}}{1+\abs{y}}}^{\beta_1}
\ssgrklam{\frac{\dist(y)}{1+\dist(y)}}^{\beta_2}e^{-\sigma\abs{x-y}^2}\dd y\\
&\leq 
C_{\beta_1,\beta_2}
\int_{\bR^2}e^{-\sigma\abs{x-y}^2}\dd y
=
C_{\beta_1,\beta_2} \int_{\bR^2}e^{-\sigma\abs{y}^2}\dd y =C_{\beta_1,\beta_2,\sigma}<\infty,
\end{align*}
which, together with~\eqref{eq:intb1b2:1}, proves the assertion.
\end{proof}

Next, we prove an estimate on $\bR^d$ with $d\in\bN$, which has~\eqref{eq:b1b2a1a2s} with $\alpha_2=\beta_2=0$ as a special case.

\begin{lemma}\label{lem:b1a1s}
For all $d\in\bN$, all $\sigma>0$, all $\alpha\in\bR$, and all $\beta>-d$, there exists a constant $C=C(d,\sigma,\alpha,\beta)<\infty$, such that
\begin{align*}
\sup_{x\in\bR^d}
\int_{\bR^d}
\ssgrklam{\frac{\abs{y}}{1+\abs{y}}}^{\beta}
\ssgrklam{\frac{1+\abs{y}}{1+\abs{x}}}^{\alpha}
e^{-\sigma\abs{x-y}^2}\dd y
\leq C.
\end{align*}
\end{lemma}

\begin{proof}
Throughout this proof, the letter $C$ is used for a finite positive constant that does not depend on $x$. This constant may vary at any appearance. 
Note that, due to Hölder's inequality, it is enough to prove the statement for $\alpha=0$ and $\beta>-d$ on the one hand and for $\alpha\in\bR$ and $\beta=0$ on the other hand.
We start with the first case:

\smallskip

\noindent \emph{Case 1. $\alpha=0$ and $\beta>-d$.} In this case, we have to show that
\[
\int_{\bR^d}
e^{-\sigma\abs{x-y}^2}
\ssgrklam{\frac{\abs{y}}{1+\abs{y}}}^\beta
\,\dy
\leq C.
\]
We argue as follows:
\begin{align*}
\int_{\bR^d}
e^{-\sigma\abs{x-y}^2}
&\ssgrklam{\frac{\abs{y}}{1+\abs{y}}}^\beta
\,\dy\\
&=
\int_{\abs{y}\leq 1}
e^{-\sigma\abs{x-y}^2}
\ssgrklam{\frac{\abs{y}}{1+\abs{y}}}^\beta
\,\dy
+
\int_{\abs{y}\geq 1}
e^{-\sigma\abs{x-y}^2}
\ssgrklam{\frac{\abs{y}}{1+\abs{y}}}^\beta
\,\dy\\
&\leq 
C_\beta
\int_{\abs{y}\leq 1}
e^{-\sigma\abs{x-y}^2}
\abs{y}^\beta
\,\dy
+
C_\beta
\int_{\abs{y}\geq 1}
e^{-\sigma\abs{x-y}^2}
\,\dy\\
&\leq 
C_\beta
\int_{\abs{y}\leq 1}
\abs{y}^\beta
\,\dy
+
C_\beta
\int_{\bR^d}
e^{-\sigma\abs{x-y}^2}
\,\dy
\leq
C_{\beta,\sigma,d}<\infty,
\end{align*}
since $\beta>-d$ and the value of the last integral does not depend on $x$.

\smallskip

\noindent\emph{Case 2. $\beta=0$.} We distinguish two sub-cases. 

\smallskip

\noindent\emph{Case 2.1. $\beta=0$ and $\alpha>0$.}
Recall that for arbitrary $\alpha>0$ and $a,b\geq 0$,
\[
2^{(1\land \alpha)-1}(a^\alpha+b^\alpha)
\leq
(a+b)^\alpha
\leq
2^{(1\lor \alpha)-1}(a^\alpha+b^\alpha).
\] 
Using this, we obtain
\begin{align*}
\int_{\bR^d}
e^{-\sigma\abs{x-y}^2}
&
\grklam{1+\abs{y}}^\alpha
\,\dy\\
&\leq
C_\alpha
\int_{\bR^d}
e^{-\sigma\abs{x-y}^2}
\geklam{\grklam{1+\abs{x}}^\alpha + \abs{y-x}^\alpha}
\,\dy\\
&\leq
C_\alpha
\grklam{1+\abs{x}}^\alpha 
\int_{\bR^d}
e^{-\sigma\abs{x-y}^2}
\,\dy
+
C_\alpha
\int_{\bR^d}
\abs{y-x}^\alpha
e^{-\sigma\abs{x-y}^2}
\,\dy\\
&\leq
C_{\alpha} \grklam{1+\abs{x}}^\alpha
\int_{\bR^d}
e^{-\sigma\abs{y}^2}
\,\dy
+
C_\alpha
\int_{\bR^d}
\abs{y}^\alpha
e^{-\sigma\abs{y}^2}
\,\dy
\leq C_{\alpha,\sigma,d},
\end{align*}
since $\ggklam{y\mapsto\abs{y}}\in L_\alpha(\bR^d,e^{-\sigma\abs{y}^2}\mathrm dy)$ for all $\alpha\geq 0$.

\smallskip

\noindent\emph{Case 2.2. $\beta=0$ and $\alpha<0$.} This is the rather difficult case. We start by splitting the integral the following way:
\begin{align*}
\int_{\bR^d}
e^{-\sigma\abs{x-y}^2}
\grklam{1+\abs{y}}^\alpha
\,\dy
=
\ssgrklam{\int_{\abs{y}\leq \frac{\abs{x}}{2}} + \int_{\abs{y}\geq \frac{\abs{x}}{2}}}
e^{-\sigma\abs{x-y}^2}
\grklam{1+\abs{y}}^\alpha
\,\dy.
\end{align*}
Since $\alpha<0$, the second integral on the right hand side can be estimated as follows:
\begin{align*}
\int_{\abs{y}\geq \frac{\abs{x}}{2}}
e^{-\sigma\abs{x-y}^2}
\grklam{1+\abs{y}}^\alpha
\,\dy
&\leq
C_\alpha
\grklam{1+\abs{x}}^\alpha
\int_{\abs{y}\geq \frac{\abs{x}}{2}}
e^{-\sigma\abs{x-y}^2}
\,\dy\\
&\leq
C_\alpha
\grklam{1+\abs{x}}^\alpha
\int_{\bR^d}
e^{-\sigma\abs{y}^2}
\,\dy
\leq
C_{\alpha,\sigma,d}(1+\abs{x})^\alpha.
\end{align*}
Since $\abs{y-x}\geq \abs{x}/2$ if $\abs{y}\leq \abs{x}/2$, the other part of the integral ($\abs{y}\leq \abs{x}/2$) can be treated as follows:
\begin{align*}
\int_{\abs{y}\leq \frac{\abs{x}}{2}}
e^{-\sigma\abs{x-y}^2}
\grklam{1+\abs{y}}^\alpha
\,\dy
&\leq
\int_{\abs{y}\leq \frac{\abs{x}}{2}}
e^{-\sigma\abs{x-y}^2}
\,\dy\\
&\leq
\int_{\abs{y}\leq \frac{\abs{x}}{2}}
e^{-\sigma\frac{\abs{x}^2}{4}}
\,\dy\\
&=
e^{-\sigma\frac{\abs{x}^2}{4}}
\int_{\abs{y}\leq \frac{\abs{x}}{2}}
\,\dy\\
&\leq
C_{d}\, \abs{x}^d \,e^{-\sigma\frac{\abs{x}^2}{4}}
\leq
C_{d}\cdot
e^{-\sigma\frac{\abs{x}^2}{4}}
\grklam{1+\abs{x}^d}. 
\end{align*}
Thus,
\begin{align*}
\int_{\abs{y}\leq \frac{\abs{x}}{2}}
&e^{-\sigma\abs{x-y}^2}
\ssgrklam{\frac{1+\abs{y}}{1+\abs{x}}}^\alpha
\,\dy\\
&=
\int_{\abs{y}\leq \frac{\abs{x}}{2}}
e^{-\sigma\abs{x-y}^2}
\grklam{1+\abs{y}}^\alpha
\,\dy
\cdot
\grklam{1+\abs{x}}^{-\alpha}
\\
&\leq
C_{d}
\cdot
e^{-\sigma\frac{\abs{x}^2}{4}}
\grklam{1+\abs{x}^d}
\grklam{1+\abs{x}}^{-\alpha}
\leq C_d\cdot e^{-\sigma\frac{\abs{x}^2}{4}}\grklam{1+\abs{x}}^{d-\alpha}.
\end{align*}
Since $\alpha<0$, the second term on the right hand side is continuous w.r.t.\ $x\in\bR^d$ and it converges to zero for $\abs{x}\to\infty$; thus, it is bounded by a constant that depends only on $\alpha$, $d$, and $\sigma$.
\end{proof}

The last piece needed in the proof of Lemma~\ref{lem:b1b2a1a2s} is the special case $\beta_1=\beta_2=\alpha_1=0$ in~\eqref{eq:b1b2a1a2s}.

\begin{lemma}\label{lem:a2s}
For all $\sigma>0$ and all $\alpha\in\bR$, there exists a constant $C=C(\sigma,\alpha)<\infty$, such that
\begin{align*}
\sup_{x\in\domaina}
\int_{\domaina}
\ssgrklam{\frac{1+\dist(y)}{1+\dist(x)}}^{\alpha}
e^{-\sigma\abs{x-y}^2}\dd y
\leq C.
\end{align*}
\end{lemma}

\begin{proof}
We distinguish two cases.

\smallskip

\noindent\emph{Case 1. $\alpha>0$.}
We split the domain into the following subdomains
\begin{align*}
\domaina_1&:=\ggklam{x=(r\cos(\vartheta),r\sin(\vartheta))\in\bR^2\colon 0<r<\infty,0<\vartheta<\grklam{\tfrac{\pi}{2}\land\tfrac{\kappa_0}{2}}}\\
\domaina_2&:=\ggklam{x=(r\cos(\vartheta),r\sin(\vartheta))\in\bR^2\colon 0<r<\infty,\grklam{\tfrac{\pi}{2}\land\tfrac{\kappa_0}{2}}<\vartheta<\kappa_0-\grklam{\tfrac{\pi}{2}\land\tfrac{\kappa_0}{2}}}\\
\domaina_3&:=\ggklam{x=(r\cos(\vartheta),r\sin(\vartheta))\in\bR^2\colon 0<r<\infty,\kappa_0-\grklam{\tfrac{\pi}{2}\land\tfrac{\kappa_0}{2}}<\vartheta<\kappa_0}
\end{align*}
and distinguish three sub-cases.

\smallskip

\noindent\emph{Case 1.1. $x\in\domaina_2$.} In this case, $\rho(x)=\abs{x}$, so that the assertion follows from Lemma~\ref{lem:b1a1s} since $\rho(y)\leq\abs{y}$ for all $y\in\domaina$.

\smallskip

\noindent\emph{Case 1.2. $x\in\domaina_1$.} First note that $\rho(x)=x_2$ for $x=(x_1,x_2)\in\domaina_1$. If $\kappa_0\leq \pi$, then 
\begin{align*}
\int_{\domaina}
\grklam{1+\dist(y)}^{\alpha}
e^{-\sigma\abs{x-y}^2}\dd y
&=
2
\int_{\domaina_1}
\grklam{1+\dist(y)}^{\alpha}
e^{-\sigma\abs{x-y}^2}\dd y\\
&\leq 
2\int_0^\infty
\int_0^\infty
\grklam{1+y_2}^\alpha e^{-\sigma\abs{x_2-y_2}^2}\dd y_2\cdot e^{-\sigma\abs{x_1-y_1}^2}\dd y_1,
\end{align*} 
so that an application of Lemma~\ref{lem:b1a1s} with $d=1$ and $\beta=0$ yields
\begin{align*}
\int_{\domaina_1}
\grklam{1+\dist(y)}^{\alpha}
e^{-\sigma\abs{x-y}^2}\dd y
\leq
C_{\alpha,\sigma} \grklam{1+x_2}^\alpha
\int_\bR
e^{-\sigma\abs{y_1}^2}\dd y_1
\leq 
C_{\alpha,\sigma}\grklam{1+\dist(x)}^\alpha.
\end{align*}
If $\kappa_0>\pi$ we split the integral the following way (convention: $\bR^d_-:=\{x\in\bR^d\colon x_1<0\}$ and $\bR^2_+:=\{x\in\bR^d\colon x_1>0\}$):
\begin{align*}
\int_{\domaina}&
\grklam{1+\dist(y)}^{\alpha}
e^{-\sigma\abs{x-y}^2}\dd y\\
&=
\ssgrklam{
\int_{\domaina_1}
+
\int_{(\domaina_2\cup\domaina_3)\cap\bR^2_-}
+
\int_{\domaina_3\cap \bR^2_+}}
\grklam{1+\dist(y)}^{\alpha}
e^{-\sigma\abs{x-y}^2}\dd y\\
&=:I+I\!I+I\!I\!I.
\end{align*}
Similar arguments as in the case $\kappa_0\leq \pi$ lead to
\begin{align*}
I
\leq
C_{\alpha,\sigma} \grklam{1+x_2}^\alpha
\int_\bR
e^{-\sigma\abs{y_1}^2}\dd y_1
\leq 
C_{\alpha,\sigma}\grklam{1+\dist(x)}^\alpha.
\end{align*}
For estimating $I\!I$ we first note that for all $x_1>0$, $y_1\in\bR$, that
\[
\one_{(-\infty,0)}(y_1)\cdot e^{-\sigma\abs{x_1-y_1}^2}
\leq
\one_{(-\infty,0)}(y_1)\cdot e^{-\sigma\abs{y_1}^2}.
\]
Thus, by Lemma~\ref{lem:b1a1s} with $d=2$ and $\beta=0$,
\begin{align*}
I\!I
&=
\int_{(\domaina_2\cup\domaina_3)\cap\bR^2_-}
\grklam{1+\dist(y)}^{\alpha}
e^{-\sigma\abs{x-y}^2}\dd y\\
&\leq
\int_{\bR^2_-}
\grklam{1+\abs{y}}^{\alpha}
e^{-\sigma\abs{(0,x_2)-y}^2}\dd y
\leq
C_{\alpha,\sigma}\grklam{1+\abs{(0,x_2)}}^\alpha
=
C_{\alpha,\sigma}\grklam{1+\dist(x)}^\alpha.
\end{align*}
Finally, if $\domaina_3\cap \bR^2_+\neq\emptyset$, we write $\tilde{\dist}:=\mathrm{dist}(\cdot,(\bR^1_+\times\{0\}))$ and find that
\begin{align*}
I\!I\!I
&=
\int_{\domaina_3\cap \bR^2_+}
\grklam{1+\dist(y)}^{\alpha}
e^{-\sigma\abs{x-y}^2}\dd y\\
&\leq
\int_{\domaina_3\cap \bR^2_+}
\grklam{1+\tilde{\dist}(y)}^{\alpha}
e^{-\sigma\abs{x-y}^2}\dd y\\
&\leq
\int_{0}^\infty
\int_{-\infty}^0
\grklam{1+\abs{y_2}}^\alpha e^{-\sigma\abs{x_2-y_2}^2}\dd y_2
\cdot e^{-\sigma\abs{x_1-y_1}^2}\dd y_1,
\end{align*}
so that another application of Lemma~\ref{lem:b1a1s} with $d=1$ and $\beta=0$ guarantees
\begin{align*}
I\!I\!I\leq C_{\alpha,\sigma}\grklam{1+\dist(x)}^\alpha,
\end{align*}
which together with the estimates above yields the assertion also for $x\in\domaina_1$.

\smallskip
\noindent\emph{Case 1.3. $x\in\domaina_3$.} This follows from Case 2 after suitable rotation and reflection due to the symmetry of the domain $\domaina$.

\smallskip 
\noindent\emph{Case 2. $\alpha<0$.} Let $x\in \domaina$. Then
\begin{align*}
\int_{\domaina}
\grklam{1+\dist(y)}^{\alpha}
&e^{-\sigma\abs{x-y}^2}\dd y\\
&=
\ssgrklam{
\int_{\domaina\cap\{\dist<\dist(x)/2\}}
+
\int_{\domaina\cap\{\dist\geq\dist(x)/2\}}
}\grklam{1+\dist(y)}^{\alpha}
e^{-\sigma\abs{x-y}^2}\dd y.
\end{align*}
Since $\alpha<0$, the second integral on the right hand side is easily estimated:
\begin{align*}
\int_{\domaina\cap\{\dist\geq\dist(x)/2\}
}\grklam{1+\dist(y)}^{\alpha}
e^{-\sigma\abs{x-y}^2}\dd y
&\leq
C_{\alpha}\grklam{1+\rho(x)}^\alpha\int_{\bR^2}e^{-\sigma\abs{y}^2}\dd y\\
&\leq C_{\alpha,\sigma}\grklam{1+\rho(x)}^\alpha.
\end{align*}
To estimate the first integral on the right hand side above, we first note that for $y\in\domaina$ if $\rho(y)<\rho(x)/2$, then $\abs{y-x}>\rho(x)/2$.
Thus,
\begin{align*}
\int_{\domaina\cap\{\dist<\dist(x)/2\}}
\grklam{1+\dist(y)}^{\alpha}
e^{-\sigma\abs{x-y}^2}\dd y
&\leq
C_\alpha 
\int_{\domaina\cap\{\dist<\dist(x)/2\}}
e^{-\tfrac{\sigma}{2}\abs{x-y}^2}
\cdot
e^{-\tfrac{\sigma}{2}\abs{x-y}^2}\dd y\\
&\leq
C_\alpha\,
e^{-\tfrac{\sigma}{8}\rho(x)^2}
\int_{\domaina\cap\{\dist<\dist(x)/2\}}
e^{-\tfrac{\sigma}{2}\abs{x-y}^2}
\dd y\\
&\leq C_{\alpha,\sigma}\,e^{-\tfrac{\sigma}{8}\rho(x)^2},
\end{align*}
and therefore, 
\begin{align*}
\int_{\domaina\cap\{\dist<\dist(x)/2\}}
\ssgrklam{\frac{1+\dist(y)}{1+\dist(x)}}^{\alpha}
e^{-\sigma\abs{x-y}^2}\dd y
\leq C_{\alpha,\sigma}\,e^{-\tfrac{\sigma}{8}\rho(x)^2} \grklam{1+\rho(x)}^{-\alpha}
\leq
C_{\alpha,\sigma}<\infty.
\end{align*}
All in all, the assertion is proved also for $\alpha<0$.
\end{proof}

Finally, in the proof of Theorem~\ref{thm:LpEstimate} we also make use of the following basic estimate.

\begin{lemma}\label{lem:At}
For arbitrary $\alpha>0$, 
\begin{align*}
\sup_{A>0}\int_0^\infty \frac{A^\alpha}{(A+\sqrt{s})^{\alpha+2}}\,\ds \leq \frac{2}{\alpha}<\infty.
\end{align*}
\end{lemma}

\begin{proof}
The proof follows by basic calculations and is left to the reader.
\end{proof}

\begin{appendix}

\mysection{Some properties of the spaces $\frwsob^n_{p,\Theta,\theta,0}(\domaina,T)$}
\label{sec:frH}

In this section we prove some fundamental properties of the spaces $\frwsob^n_{p,\Theta,\theta,0}(\domaina,T)$ for $2\leq p<\infty$, $n\in\bN$, and $\Theta,\theta\in\bR$. In particular, we establish completeness of these spaces by proving the following estimate.

\begin{lemma}\label{lem:frH:cont}
Let $2\leq p<\infty$, let $\Theta,\theta\in\bR$, and let $1<\Theta<p+1$. Assume that $u\in\frwsob^1_{p,\Theta,\theta,0}(\domaina,T)$ satisfies
\[
\dd u = f\dd t + g^k\dd w^k_t,\quad t\in (0,T],
\]
on $\domaina$ in the sense of distributions with $u(0,\cdot)=0$ for some $f\in\bwsob^{-1}_{p,\Theta+p,\theta+p}(\domaina,T)$ and $g\in\bL_{p,\Theta,\theta}(\domaina,T;\ell_2)$. Then
\begin{equation}\label{eq:frH:cont}
\E\sgeklam{\sup_{t\leq T}\nnrm{u(t,\cdot)}{L_{p,\Theta,\theta}(\domaina)}^p}
\leq C\,
\sgrklam{
\nnrm{u}{\bwsob^1_{p,\Theta-p,\theta-p}(\domaina,T)}^p
+
\nnrm{f}{\bwsob^{-1}_{p,\Theta+p,\theta+p}(\domaina,T)}^p
+
\nnrm{g}{\bL_{p,\Theta,\theta}(\domaina,T;\ell_2)}^p
}
\end{equation}
with $C=C(p,\Theta,\theta,\kappa_0,T)$.
\end{lemma}

For the proof of Lemma~\ref{lem:frH:cont} we will use estimates of the norm on $L_{p,\Theta,\theta}(\domaina)$ by means of suitably weighted series of weighted norms of proper localizations on a $\cont^1$-subdomain $G\subset \domaina$. To state these representations, we write for arbitrary domains $G\subset\bR^d$ with non-empty boundary $\partial G\neq\emptyset$, $1<p<\infty$, and $\Theta\in\bR$,
\[
L_{p,\Theta}(G)
:=
L_p(G,\cB(G),\dist_G(x)^{\Theta-d}\dd x;\bR),
\]
where $\dist_G:=\mathrm{dist}(\cdot,G)$.

\begin{lemma}\label{lem:LpTt:localized}
Let $u\colon\cD\to\bR$ be a measurable function, let $1<p<\infty$, and let $\Theta,\theta\in \bR$.
\begin{enumerate}[leftmargin=*,label=\textup{(\roman*)}] 
\item For every $\xi\in\cont^\infty_0((0,\infty))$, there exists a constant $C=C(\theta,p,\xi)$, such that
\[
\sum_{n\in\bZ} e^{n\theta}\int_\cD \Abs{\xi(\cdot) u(e^n\cdot)}^p \ssgrklam{\frac{\rho_\cD}{\rho_\circ}}^{\Theta-2}\,\dx
\leq C
\int_\cD \abs{u}^p \rho_\circ^{\theta-2}\ssgrklam{\frac{\rho_\cD}{\rho_\circ}}^{\Theta-2}\,\dx.
\]
\item For every $\xi\in\cont^\infty_0((0,\infty))$ and $G\subset\cD$, such that $\rho_G\sim\rho_\cD$ on $\textup{supp}\,\xi\cap\cD$, there exists a constant $C=C(\theta,p,\Theta,\xi,G)$, such that
\[
\sum_{n\in\bZ} e^{n\theta}\int_G \Abs{\xi(\cdot) u(e^n\cdot)}^p \ssgrklam{\frac{\rho_G}{\rho_\circ}}^{\Theta-2}\,\dx
\leq C
\int_\cD \abs{u}^p \rho_\circ^{\theta-2}\ssgrklam{\frac{\rho_\cD}{\rho_\circ}}^{\Theta-2}\,\dx.
\]
\item For every $\xi\in\cont^\infty_0((0,\infty))$ such that
\begin{equation}\label{eqn 5.6.544a}
\sum_{n=-\infty}^{\infty}\xi(e^{n+t})>c>0, \quad \forall \, t\in \bR,
\end{equation}
there exists a constant $C=C(\theta,p,\xi)$ such that
\[
\int_\cD \abs{u}^p \rho_\circ^{\theta-2}\ssgrklam{\frac{\rho_\cD}{\rho_\circ}}^{\Theta-2}\,\dx
\leq C
\sum_{n\in\bZ} e^{n\theta}\int_\cD \Abs{\xi(\cdot) u(e^n\cdot)}^p \ssgrklam{\frac{\rho_\cD}{\rho_\circ}}^{\Theta-2}\,\dx.
\]
\item For every $\xi\in\cont^\infty_0((0,\infty))$  and $G\subset\cD$, such that~\eqref{eqn 5.6.544a} holds and $\rho_G\sim\rho_\cD$ on $\textup{supp}\,\xi\cap\cD$, there exists a constant $C=C(\theta,p,\xi,\Theta,G)$ such that
\[
\int_\cD \abs{u}^p \rho_\circ^{\theta-2}\ssgrklam{\frac{\rho_\cD}{\rho_\circ}}^{\Theta-2}\,\dx
\leq C
\sum_{n\in\bZ} e^{n\theta}\int_G \Abs{\xi(\cdot) u(e^n\cdot)}^p \ssgrklam{\frac{\rho_G}{\rho_\circ}}^{\Theta-2}\,\dx.
\]
\end{enumerate}
In particular, for every $\xi\in\cont_0^\infty((0,\infty))$ and $G\subseteq\domaina$ such that $\dist_G\sim\dist_\domaina$ on $\mathrm{supp}\, \xi \cap\domaina$ it holds that
\begin{equation}\label{eq:LpTt:wT:Tt}
\sum_{n\in\bZ} e^{n\theta} \nnrm{\xi(\cdot)u(e^n\cdot)}{L_{p,\Theta}(G)}^p
\leq C(\Theta,\theta,p,\xi,G) 
\nnrm{u}{L_{p,\Theta,\theta}(\domaina)}^p.
\end{equation}
If, in addition,~\eqref{eqn 5.6.544a} holds, then also
\begin{equation}\label{eq:LpTt:Tt:wT}
\nnrm{u}{L_{p,\Theta,\theta}(\domaina)}^p
\leq C(\Theta,\theta,p,\xi,G) 
\sum_{n\in\bZ} e^{n\theta} \nnrm{\xi(\cdot)u(e^n\cdot)}{L_{p,\Theta}(G)}^p.
\end{equation}
\end{lemma}

\begin{proof}
This lemma can be proved by following the lines of the proof of~\cite[Lemma~4.1]{CioKimLee2019}, see also~\cite[Remark~1.3]{Kry1999c}, by noting that the ratio $\dist/\distv$ is invariant under positive dilations, i.e.,
\[
\frac{\dist(\alpha x)}{\distv(\alpha x)}
=
\frac{\dist( x)}{\distv(x)}
\quad
\text{for all}
\quad
x\in\domaina,\alpha>0.\qedhere
\]
\end{proof}

The second ingredient that we use for the proof of Lemma~\ref{lem:frH:cont} is an analogous result on $\cont^1$-domains proven in~\cite[Lemma~4.2]{CioKimLee2019}.
For the convenience of the reader we state the relevant part of~\cite[Lemma~4.2]{CioKimLee2019} below in Lemma~\ref{lem:frH:cont:C1}.
Therein we use the following notation: For an arbitrary domain $G\subseteq\bR^d$ with non-empty boundary $\partial G\neq\emptyset$ and $1<p<\infty$, we write $(\wsob^1_{p,\Theta}(G),\nnrm{\cdot}{H^1_{p,\Theta}(G)})$ for the Banach space of all (equivalence classes of) locally integrable functions $u\colon G\to\bR$ such that
\[
\nnrm{u}{H^1_{p,\Theta}(G)}^p
:=
\sum_{\abs{\alpha}\leq 1}
\nnrm{u}{L_{p,\Theta+\abs{\alpha}p}(G)}^p<\infty.
\]
Moreover, 
\[
\bL_{p,\Theta}(G,T)
:=
L_p(\Omega_T,\pred_T,\prob_T;L_{p,\Theta}(G))
\quad\text{and}\quad
\bwsob^1_{p,\Theta}(G,T)
:=
L_p(\Omega_T,\pred_T,\prob_T;\wsob^1_{p,\Theta}(G)),
\]
while
\[
L_{p,\Theta}(G;\ell_2):=L_{p}(G,\cB(G),\dist_{G}^{\Theta-2}\dx;\ell_2)
\quad\text{and}\quad
\bL_{p,\Theta}(G,T;\ell_2)
:=
L_p(\Omega_T,\pred_T,\prob_T;L_{p,\Theta}(G;\ell_2)).
\]
In analogy to $\frwsob^1_{p,\Theta,\theta,0}(\domaina,T)$, see also Lemma~\ref{lem:dual:parabolic:H1} below, we write $u\in \frwsob^{1,[0]}_{p,\Theta}(G,T)$ if there exist
$f^0\in \bL_{p,\Theta+p}(G,T)$, $f^i\in \bL_{p,\Theta}(G,T)$, $i=1,2$, and $g\in\bL_{p,\Theta}(G,T;\ell_2)$, such that
\[
\dd u =(f^0+f^i_{x^i})\dd t + g^k\dd w^k_t, \quad t\in (0,T]
\]
on $G$ in the sense of distributions with $u(0,\cdot)=0$, i.e., for all $\varphi\in\cont_0^\infty(G)$ with probability one, the equality
\begin{equation}
(u(t,\cdot),\varphi)=   \int^{t}_{0}
(f^0(s,\cdot)+f^i_{x^i}(s,\cdot),\varphi) \dd s + \sum^{\infty}_{k=1} \int^{t}_{0}
(g^k(s,\cdot),\varphi)\,  dw^k_s
\end{equation}
holds for all $t \leq T$. 

\begin{lemma}[Lemma~4.2 in~\cite{CioKimLee2019}]\label{lem:frH:cont:C1}
Let $G$ be a bounded $\cont^1$ domain, $\Theta\in\bR$, $2\leq p<\infty$, and $u\in\frwsob^{1,[0]}_{p,\Theta}(G,T)$, such that
\[
\dd u =\grklam{f^0+f^i_{x^i}}\dd t + g^k\dd w^k_t, \quad t\in(0,T],
\]
on $G$, where $f^0\in \bL_{p,\Theta+p}(G,T)$, $f^i\in\bL_{p,\Theta}(G,T)$, $i=1,2$, and $g\in\bL_{p,\Theta}(G,T;\ell_2)$.
Then there exists a constant $C=C(d,p,\Theta,T,\Theta)$, such that for all $a>0$ it holds that
\begin{equation}
\begin{alignedat}{1}
\E\sgeklam{\sup_{t\leq T}\nnrm{u(t,\cdot)}{L_{p,\Theta}(G)}^p}
\leq C\,
\sgrklam{
a\,\nnrm{u}{\bwsob^1_{p,\Theta-p}(G,T)}^p
+
&a^{-1}\nnrm{f^0}{\bL_{p,\Theta+p}(G,T)}^p\\
&+
\sum_{i=1}^2 a^{-1}\nnrm{f^i}{\bL_{p,\Theta}(G,T)}^p
+
\nnrm{g}{\bL_{p,\Theta}(G,T;\ell_2)}^p
}.
\end{alignedat}
\end{equation}
\end{lemma}

Now we are ready to prove Lemma~\ref{lem:frH:cont}.

\begin{proof}[Proof of Lemma~\ref{lem:frH:cont}]
We show that for arbitrary $f^0\in\bL_{p,\Theta+p,\theta+p}(\domaina,T)$ and $f^i\in\bL_{p,\Theta,\theta}(\domaina,T)$, $i=1,2$, such that $f=f^0+f^1_{x^1}+f^2_{x^2}$ the estimate
\begin{equation}\label{eq:frH:cont:f0fi}
\begin{alignedat}{1}
\E\sgeklam{\sup_{t\leq T}\nnrm{u(t,\cdot)}{L_{p,\Theta,\theta}(\domaina)}^p}
\leq C\,
\sgrklam{
\nnrm{u}{\bwsob^1_{p,\Theta-p,\theta-p}(\domaina,T)}^p
+
&\nnrm{f^0}{\bL_{p,\Theta+p,\theta+p}(\domaina,T)}^p\\
&+
\sum_{i=1}^2 \nnrm{f^i}{\bL_{p,\Theta,\theta}(\domaina,T)}^p
+
\nnrm{g}{\bL_{p,\Theta,\theta}(\domaina,T;\ell_2)}^p
}
\end{alignedat}
\end{equation}
holds with a constant $C=C(p,\Theta,\theta,\kappa_0,T)$. Then Estimate~\eqref{eq:frH:cont} follows by Lemma~\ref{lem:dual:parabolic:H1}.

Let $U_1:=\{x\in\domaina\colon 1<\abs{x}<4\}$ and $V_1:=\{1/2<\abs{x}<8\}$.
Fix a constant $\varepsilon \in (0,1/4)$, and  for $k\in\{1,2,3\}$, let
$$U^{(k)}_1:=\{x\in \cD: 2^{-k\varepsilon}<|x|<2^{2+k\varepsilon}\}.
$$
Choose a $\cont^{\infty}$ radial non-negative  function $\eta=\eta(|x|)$ such that $\eta(t)=1$ for $t\in [1,4]$ and $\eta(t)=0$ if $t\not\in [2^{-\varepsilon}, 2^{2+\varepsilon}]$. Also choose a $\cont^{1}$ domain $G\subset\cD$ such that
$$
U^{(2)}_1 \subset G \subset U^{(3)}_1\subset V_1.
$$
Note that, in particular,
\[
\sum_{n\in\bZ} \eta(e^{n+t})\geq 1>0,\quad t\in\bR.
\]
By the choice of $\eta$ and $G$, there exists $C=C(\varepsilon)$ such that for all $x\in \cD \cap  \text{supp}\,\eta$,
\begin{equation}
 \label{rho}
C^{-1} \rho_{\cD}(x)\leq \rho_G (x) \leq C \rho_{\cD}(x).
\end{equation}
By Lemma~\ref{lem:LpTt:localized},
\begin{equation}
\E \sgeklam{\sup_{t\leq T} \nnrm{u(t,\cdot)}{L_{p,\Theta,\theta}(\cD)}^p}
\leq C
\sum_{n\in \bZ} e^{n\theta} \,\E \sgeklam{\sup_{t\leq T} \gnnrm{\eta(x)u(t,e^nx)}{L_{p,\Theta}(G)}^p}.
\end{equation}
For $n\in\bZ$, let $v_n(t,x):=\eta(x) u(t,e^n x)$. Then 
\[
\dd v_n(t,x)
=
\geklam{ e^{-n} \eta(x)(f^i(t,e^nx))_{x^i}+\eta(x)f^0(t,e^nx)}\dd t
+
\eta(x)g^k(t,e^nx)\dd w^k_t, \quad t\in(0,T],
\]
on $G$ with $v_n(0,\cdot)=0$.
Note that for $i=1,2$ and $n\in\bZ$,
\[
e^{-n} \eta(x) (f^i(t,e^nx))_{x^i}
=  
e^{-n} [\eta(x) f^i(t,e^nx)]_{x^i}
-
e^{-n}\eta_{x^i}(x)f^i(t,e^nx),
\]
and
\begin{equation}\label{eq_sec4_2}
(v_n)_{x^i}(t,x)=e^{n}\eta(x) u_{x^i}(t,e^nx)-\eta_{x^i}(x)u(t,e^nx).
\end{equation}
Thus, if for all $n\in\bN$ it holds that 
\begin{equation}\label{eq:LpTt:localized:vn}
\left.
\begin{alignedat}{1}
v_n(t,x)&:=\eta(x)u(t,e^nx)\in \bH^1_{p,\Theta-p}(G,T),\\
f^{0,n}&:=\eta(x)f^0(t,e^nx) - e^{-n}\eta_{x^i}(x)f^i(t,e^nx)  \in \bL_{p,\Theta+p}(G,T),\\
f^{i,n}&:= e^{-n} \eta(x) f^i(t,e^nx) \in \bL_{p,\Theta}(G,T),\quad i=1,2,\qquad \text{and}\\
(g^{k,n})_{k\in\bN}&:=(\eta(x)g^k(t,e^n x))_{k\in\bN}\in\bL_{p,\Theta}(G,T;\ell_2),
\end{alignedat}
\right\}
\end{equation}
then we may apply Lemma~\ref{lem:frH:cont:C1} and obtain for arbitrary $a>0$ that
\begin{align*}
\E&\sgeklam{\sup_{t\leq T} \gnnrm{v_n(t,\cdot)}{L_{p,\Theta}(G)}^p}\\
&\leq C
\ssgrklam{
a\, \nnrm{\eta(x)u(t,e^n x)}{\bL_{p,\Theta-p}(G,T)}^p
+
a\, \sum_{i} e^{np}\nnrm{\eta(x)u_{x^i}(t,e^nx)}{\bL_{p,\Theta}(G,T)}^p\\
&\qquad\qquad
+ 
a \,\sum_{i}\nnrm{\eta_{x^i}(x)u(t,e^nx)}{\bL_{p,\Theta}(G,T)}^p
+
a^{-1}\sum_{i}\nnrm{\eta(x)f^0(t,e^nx)}{\bL_{p,\Theta+p}(G,T)}^p\\
&\qquad\qquad
+
a^{-1}e^{-np}\sum_{i}\nnrm{\eta_{x^i}(x)f^i(t,e^nx)}{\bL_{p,\Theta+p}(G,T)}^p
+
a^{-1}e^{-np}\sum_{i}\nnrm{\eta(x)f^i(t,e^nx)}{\bL_{p,\Theta}(G,T)}^p\\
&\qquad\qquad
+
\nnrm{\eta(x)g^k(t,e^nx)}{\bL_{p,\Theta}(G,T;\ell_2)}^p
},
\end{align*}
and, as a consequence (set $a=e^{-np}>0$ for $n\in\bZ$), 
\begin{align*}
\E&\sgeklam{\sup_{t\leq T}\nnrm{u(t,\cdot)}{L_{p,\Theta,\theta}(\domaina)}^p}\\
&\leq C\,
\ssgrklam{
\sum_{n} e^{n(\theta-p)}\, \nnrm{\eta(x)u(t,e^n x)}{\bL_{p,\Theta-p}(G,T)}^p
+
\sum_{n,i}e^{n\theta}\nnrm{\eta(x)u_{x^i}(t,e^nx)}{\bL_{p,\Theta}(G,T)}^p\\
&\qquad \qquad
+ 
\sum_{n,i}e^{n(\theta-p)} \,\nnrm{\eta_{x^i}(x)u(t,e^nx)}{\bL_{p,\Theta}(G,T)}^p
+
\sum_{n}e^{n(\theta+p)}\nnrm{\eta(x)f^0(t,e^nx)}{\bL_{p,\Theta+p}(G,T)}^p\\
&\qquad\qquad
+
\sum_{n,i}e^{n\theta}\nnrm{\eta_{x^i}(x)f^i(t,e^nx)}{\bL_{p,\Theta+p}(G,T)}^p
+
\sum_{n,i}e^{n\theta}\nnrm{\eta(x)f^i(t,e^nx)}{\bL_{p,\Theta}(G,T)}^p\\
&\qquad\qquad
+
\sum_{n}e^{n\theta}\nnrm{\eta(x)g^k(t,e^nx)}{\bL_{p,\Theta}(G,T;\ell_2)}^p
}.
\end{align*}
Using Lemma~\ref{lem:LpTt:localized}, in particular~\eqref{eq:LpTt:wT:Tt}, 
the terms on the right hand side can be estimated by a finite constant times
\begin{align*}
\nnrm{u}{\bL_{p,\Theta-p,\theta-p}(\domaina,T)}^p
+
\nnrm{u_{x^i}}{\bL_{p,\Theta,\theta}(\domaina,T)}^p
&+
\nnrm{u}{\bL_{p,\Theta,\theta-p}(\domaina,T)}^p
+
\nnrm{f^0}{\bL_{p,\Theta+p,\theta+p}(\domaina,T)}^p\\
&+
\nnrm{f^i}{\bL_{p,\Theta+p,\theta}(\domaina,T)}^p
+
\nnrm{f^i}{\bL_{p,\Theta,\theta}(\domaina,T)}^p
+
\nnrm{g}{\bL_{p,\Theta,\theta}(\domaina,T;\ell_2)}^p.
\end{align*}
In particular, this shows that~\eqref{eq:LpTt:localized:vn} holds for all $n\in\bZ$, and 
 implies~\eqref{eq:frH:cont:f0fi} since $L_{p,\Theta_1,\theta_0}(\domaina)\hookrightarrow L_{p,\Theta_2,\theta_0}(\domaina)$ for all $\theta_0\in\bR$ and $\Theta_1\leq\Theta_2$.
\end{proof}

\begin{corollary}
Let $1<p<\infty$, $n\in\bN$, and $\Theta,\theta\in\bR$. Then $(\frH^n_{p,\Theta,\theta,0}(\domaina,T),\nnrm{\cdot}{\frH^n_{p,\Theta,\theta,0}(\domaina,T)})$ is a Banach space.
\end{corollary}

\begin{proof}
This follows from the fact that the the spaces $\bwsob^n_{p,\Theta,\theta}(\domain,T)$ and $\bwsob^n_{p,\Theta,\theta}(\domain,T;\ell_2)$ are separable Banach spaces, by using Lemma~\ref{lem:frH:cont}. 
\end{proof}

\mysection{Some properties of weighted Sobolev space of negative order}\label{sec:dualspace}

In this section we analyse the duality relationships between the weighted Sobolev spaces introduced in Section~\ref{sec:MaxReg}. We first prove the following fundamental result.

\begin{lemma}\label{lem:wSob:test:dense}
Let $1<p<\infty$ and $\Theta,\theta\in \bR$. Then $\cont^\infty_0(\domaina)\subseteq H^1_{p,\Theta,\theta}(\domaina)$ dense. 
\end{lemma}

\begin{proof}
This is a consequence of the fact that $\cont_0^\infty(\domaina)\subseteq H^1_{p,\Theta}(\domaina)$ for all $\Theta\in\bR$. The details are left to the reader.
\end{proof}

\begin{lemma}\label{lem:dual:Lo}
Let $1<p<\infty$ and $\Theta,\theta\in \bR$. Moreover, let $p'>1$ and $\Theta',\theta'\in\bR$ be such that
\[
\frac{1}{p}+\frac{1}{p'}=1,\quad \frac{\theta}{p}+\frac{\theta'}{p'}=\frac{\Theta}{p}+\frac{\Theta'}{p'}=2.
\]
Then the mapping
\begin{align*}
I\colon L_{p',\Theta',\theta'}(\cD) &\to (L_{p,\Theta,\theta}(\cD))^*\\
f & \mapsto If:=\int_{\cD} f\cdot\,\dx,
\end{align*}
is an isometric isomorphism.
\end{lemma}

\begin{proof}
First we check that $I$ is well-defined and that
\[
\nnrm{If}{(L_{p,,\Theta,\theta}(\cD))^*}
\leq
\nnrm{f}{L_{p',\Theta',\theta'}(\cD)},
\qquad
f\in L_{p',\Theta',\theta'}(\cD).
\]
But this follows trivially by H\"older's inequality, since
\[
\frac{\Theta-2}{p}+\frac{\Theta'-2}{p'}
=
\frac{\theta-2}{p}+\frac{\theta'-2}{p'}
=
2-2\cdot 1 =0,
\]
so that
\begin{align*}
\int_\cD f\,g\,\dx
=
\int_\cD 
f\, \distv^{\frac{\theta'-2}{p'}} \ssgrklam{\frac{\dist}{\distv}}^{\frac{\Theta'-2}{p'}} 
g\,\distv^{\frac{\theta-2}{p}} \ssgrklam{\frac{\dist}{\distv}}^{\frac{\Theta-2}{p}}\,\dx
\leq
\nnrm{f}{L_{p',\Theta',\theta'}(\cD)}
\nnrm{g}{L_{p,\Theta,\theta}(\cD)}.
\end{align*}
Therefore,
\[
\nnrm{If}{(L_{p,\Theta,\theta}(\cD))^*}
=
\sup_{\substack{ g\in L_{p,\Theta,\theta}(\cD)\\ \nnrm{g}{L_{p,\Theta,\theta}(\cD)}=1}}
(If)(g)
\leq
\nnrm{f}{L_{p',\Theta',\theta'}(\cD)},
\qquad
f\in L_{p',\Theta',\theta'}(\cD).
\]
Moreover, for every $f\in L_{p',\Theta',\theta'}(\cD)\setminus\{0\}$, if we choose
\[
g:=\textup{sign}(f)\ssgrklam{\frac{\abs{f}}{\nnrm{f}{L_{p',\Theta',\theta'}(\cD)}}}^{\frac{p'}{p}}
\distv^{\theta'-2}
\ssgrklam{\frac{\dist}{\distv}}^{\Theta'-2},
\]
then, since
\[
p(\theta'-2)+\theta-2
=
(\theta'-2) + p\ssgrklam{\frac{\theta'-2}{p'}+\frac{\theta-2}{p}}
=
\theta'-2
\quad\text{and}\quad
p(\Theta'-2)+\Theta-2
=
\Theta'-2,
\]
we have
\begin{align*}
\nnrm{g}{L_{p,\Theta,\theta}(\cD)}^p
=
1
\end{align*}
and
\[
\int_\cD f(x)g(x)\,\dx
=
\nnrm{f}{L_{p',\Theta',\theta'}(\cD)}.
\]
Thus, 
\[
\nnrm{If}{(L_{p,\Theta,\theta}(\cD))^*}
=
\nnrm{f}{L_{p',\Theta',\theta'}(\cD)},\qquad f\in L_{p',\Theta',\theta'}(\cD),
\]
which means that the linear mapping $I$ is a well-defined isometry. 
To prove that $I$ is surjective, we apply Riesz' representation theorem, which gives us that for every $v\in (L_{p,\Theta,\theta}(\domaina))^*=(L_p(\cD,\rho_\circ(x)^{\theta-2}(\dist/\distv)^{\Theta-2}\,\dx))^*$ there exists a function $\tilde{f}\in L_{p'}(\cD,\rho_\circ(x)^{\theta-2}(\dist/\distv)^{\Theta-2}\,\dx)$, such that
\[
v(\varphi)
=
\int_\cD \tilde{f}\,\varphi\,\distv^{\theta-2}\ssgrklam{\frac{\dist}{\distv}}^{\Theta-2}\,\dx,\qquad \varphi\in L_{p,\Theta,\theta}(\cD).
\] 
Set $f:=\tilde{f}\distv^{\theta-2}(\dist/\distv)^{\Theta-2}$. Then, $f\in L_{p',\Theta',\theta'}(\cD)$, and
\[
v(\varphi)
=
\int_\cD f(x)\varphi(x)\,\dx,\qquad \varphi\in L_{p,\Theta,\theta}(\cD).
\]
Thus, $I\colon L_{p',\theta'}^{[\circ]}(\cD)\to (L_{p,\theta}^{[\circ]}(\cD))^*$ is surjective.
Since it is also an isometry, $I$ is an isometric isomorphism.
\end{proof}

The `parabolic version' of Lemma~\ref{lem:dual:Lo} reads as follows.

\begin{lemma}\label{lem:dual:parabolic:Lo}
$1<p<\infty$ and $\Theta,\theta\in \bR$. Moreover, let $p'>1$ and $\Theta',\theta'\in\bR$ be such that
\[
\frac{1}{p}+\frac{1}{p'}=1,\quad \frac{\theta}{p}+\frac{\theta'}{p'}=\frac{\Theta}{p}+\frac{\Theta'}{p'}=2.
\]
Then the mapping
\begin{align*}
I\colon \bL_{p',\Theta',\theta'}(\cD,T) &\to (\bL_{p,\Theta,\theta}(\cD,T))^*\\
f & \mapsto If:=\int_{\Omega_T}\int_{\cD} f\cdot\,\dx\,\mathrm{d}\prob_T,
\end{align*}
is an isometric isomorphism.
\end{lemma}

\begin{proof}
This follows from Lemma~\ref{lem:dual:Lo} and the fact that $L_{p,\Theta,\theta}(\cD)$ is reflexive for $1<p<\infty$, as for every reflexive Banach space $E$ and every measure space $(S,\cA,\mu)$, it is known that $L_{p'}(S;E^*)\cong (L_p(S;E))^*$, the isometric isomorphism being given by $f\mapsto\int_S \langle f,\cdot \rangle_{E^*\times E}\mathrm{d}\mu$, see, e.g.,~\cite[Corollary~1.3.22]{HytNeeVer+2016}.
\end{proof}

Let $\mathscr{D}'(\cO)$ define the space of all Schwartz distributions/generalized functions on a domain $\cO\subseteq \bR^d$.

\begin{lemma}\label{lem:dual:K1}
For $1<p<\infty$ and $\Theta,\theta\in\bR$, consider
\[
\Lambda^{-1}_{p,\Theta,\theta}(\cD)
:=
\sggklam{
g\in\mathscr{D}'(\cD)\,\colon\,
g=\sum_{\abs{\alpha}\leq 1}(-1)^{\abs{\alpha}} D^\alpha v_\alpha,\, v_\alpha\in L_{p,\Theta-\abs{\alpha}p,\theta-\abs{\alpha}p}(\cD),\abs{\alpha}\leq 1
},
\]
endowed with the norm
\begin{align*}
\nnrm{g}{\Lambda^{-1}_{p,\Theta,\theta}(\cD)}
:=
\inf
\sggklam{\sum_{\abs{\alpha}\leq 1} \nnrm{v_\alpha}{L_{p,\Theta-\abs{\alpha}p,\theta-\abs{\alpha}p}(\cD)}
}
=
\min
\sggklam{\sum_{\abs{\alpha}\leq 1} \nnrm{v_\alpha}{L_{p,\Theta-\abs{\alpha}p,\theta-\abs{\alpha}p}(\cD)}
},
\end{align*}
where the infimum is taken over all possible tuples $(v_\alpha)_{\abs{\alpha}\leq 1}$ with $v_{\alpha}\in L_{p,\Theta-\abs{\alpha}p,\theta-\abs{\alpha}p}(\cD)$, $\abs{\alpha}\leq 1$, such that $g=\sum_{\abs{\alpha}\leq 1}(-1)^\alpha D^\alpha v_\alpha$.
Then, the mapping
\begin{align*}
J\colon \wsob^{-1}_{p,\Theta,\theta}(\cD) &\to \Lambda^{-1}_{p,\Theta,\theta}(\cD)\\
f&\mapsto f|_{\cont_0^\infty(\cD)},
\end{align*}
is an isometric isomorphism. In this sense, we can say that
\begin{align*}
\wsob^{-1}_{p,\Theta,\theta}(\cD)
=
\sggklam{f^0-f^1_{x^1}-f^2_{x^2}\,\colon\, f^0\in L_{p,\Theta,\theta}(\cD),\, f^i\in L_{p,\Theta-p,\theta-p}(\cD),\,i=1,2},
\end{align*}
and that
\[
\nnrm{f}{\wsob^{-1}_{p,\Theta,\theta}(\cD)}
=
\inf\sggklam{\nnrm{f^0}{L_{p,\Theta,\theta}(\cD)}+ \nnrm{f^i}{L_{p,\Theta-p,\theta-p}(\cD)}},
\]
the infimum being taken over all possible representations of $f|_{\cont^\infty_0(\cD)}$ as $f^0+f^i_{x^i}$ with $ f^0\in L_{p,\Theta,\theta}(\cD)$ and $f^i\in L_{p,\Theta-p,\theta-p}(\cD)$, $i=1,2$.

In particular, the differentiation operators $D^{\alpha}\colon L_{p,\Theta,\theta}(\cD)\to \wsob^{-1}_{p,\Theta+p,\theta+p}(\cD)$, $\abs{\alpha}=1$, are bounded.

\end{lemma}

\begin{proof}
It is easy to see that the restriction of an arbitrary $f\in \wsob^{-1}_{p,\Theta.\theta}(\cD)$ to $\cont^\infty_0(\cD)$ is a generalized function, i.e., that $f|_{\cont^\infty_0(\cD)}\in\mathscr{D}'(\cD)$ for all $f\in \wsob^{-1}_{p,\Theta,\theta}(\cD)$.
However, we need a little bit of work to see that $f|_{\cont^\infty_0(\cD)}=f^0+f^1_{x^1}+f^2_{x^2}$ with suitable $f^0\in L_{p,\Theta,\theta}(\cD)$ and $f^i\in L_{p,\Theta-p,\theta-p}(\cD)$, $i=1,2$, and that the norms above are indeed norms, that the infimum is indeed a minimum and that the norms coincide.

\medskip

\noindent\textit{Step 1.} We first show that $J$ is well-defined, i.e., that for every $v^*\in \wsob^{-1}_{p,\Theta,\theta}(\cD)$, there exist $ f^0\in L_{p,\Theta,\theta}(\cD)$ and $f^i\in L_{p,\Theta-p,\theta-p}(\cD)$, $i=1,2$, such that $v^*|_{\cont^\infty_0(\cD)}=f^0+f^1_{x^1}+f^2_{x^2}$.
This can be seen as follows: The mapping
\begin{align*}
P\colon \wsob^{1}_{p',\Theta',\theta'}(\cD) &\to V^1_{p',\Theta'\theta'}(\cD):=L_{p',\Theta',\theta'}(\cD) \oplus_{p'} L_{p',\Theta'+p',\theta'+p'}(\cD) \oplus_{p'} L_{p',\Theta'+p',\theta'+p'}(\cD)\\
\varphi&\mapsto (\varphi,\varphi_{x^1},\varphi_{x^2}),
\end{align*}
is obviously a linear isometry.
In particular, since $\wsob^{1}_{p',\Theta',\theta'}(\cD)$ is complete, the range $\mathcal RP$ of $\wsob^{1}_{p',\Theta',\theta'}(\cD)$ under $P$ is a closed subspace of $V^1_{p',\Theta',\theta'}(\cD)$; we write $W^1_{p',\Theta',\theta'}(\cD):=\mathcal RP$.
Let $v^*\in \wsob^{-1}_{p,\Theta,\theta}(\cD)=(\wsob^{1}_{p',\Theta',\theta'}(\cD))^*$.
Define $\hat{v}:=v^*\circ P^{-1}\in (W^1_{p',\Theta',\theta'}(\cD))^*$.
Then, by Hahn-Banach Theorem, there exists a norm preserving extension $\tilde{v}\in (V^1_{p',\Theta',\theta'}(\cD))^*$ of $\hat{v}$ to $V^1_{p',\Theta',\theta'}(\cD)$.
Due to Lemma~\ref{lem:dual:Lo}, this means that there exist 
$f^0\in L_{p,\Theta,\theta}(\cD)$ and 
$f^i\in L_{p,\Theta-p,\theta-p}(\cD)$, $i=1,2$, such that
\[
\tilde{v}(\phi_0,\phi_1,\phi_2)
=
\int_\cD \phi_0 f^0 \dx
+
\int_\cD \phi_1 f^1 \dx
+
\int_\cD \phi_2 f^2 \dx,
\]
for all $\phi:=(\phi_0,\phi_1,\phi_2)\in V^1_{p',\Theta',\theta'}(\cD)$.
In particular, for every $\varphi\in \wsob^{1}_{p',\Theta',\theta'}(\cD)$, 
\[
\tilde{v}(P\varphi)
=
\tilde{v}(\varphi,\varphi_{x^1},\varphi_{x^2})
=
\int_\cD \varphi f^0 \dx
+
\int_\cD \varphi_{x^1} f^1 \dx
+
\int_\cD \varphi_{x^2} f^2 \dx.
\]
If $\varphi\in\cont^\infty_0(\cD)$, then
\[
\tilde{v}(P\varphi)
=
v^*(\varphi)
=
\int_\cD \varphi f^0 \dx
-
\int_\cD \varphi f^1_{x^1} \dx
-
\int_\cD \varphi f^2_{x^2} \dx,
\]
which is exactly the same as the application of the generalized function $f^0-f^1_{x^1}-f^2_{x^2}$ to the test function $\varphi$.

\medskip

\noindent\textit{Step 2.} $J$ is surjective, i.e., for every $g\in \mathscr{D}'(\cD)$, $g=f^0-f^1_{x^1}-f^2_{x^2}$, where $f^0\in L_{p,\Theta,\theta}(\cD)$ and $f^i\in L_{p,\Theta-p,\theta-p}(\cD)$, $i=1,2$, there exists $v^*\in \wsob^{-1}_{p,\Theta,\theta}(\cD)$, such that $v^*|_{\cont^\infty_0(\cD)}=g$.
To see this, simply define
\[
v^*(\varphi):=(f^0,\varphi) + (f^1,\varphi_{x^1}) + (f^2,\varphi_{x^2}),\quad \varphi\in  \wsob^{1}_{p',\Theta',\theta'}(\cD).
\]
Due to Lemma~\ref{lem:dual:Lo}, this mapping is well-defined and
\[
v^*(\varphi) 
\leq 
\grklam{
\nnrm{f^0}{L_{p,\Theta,\theta}(\cD)}
+
\nnrm{f^1}{L_{p,\Theta-p,\theta-p}(\cD)}
+
\nnrm{f^2}{L_{p,\Theta-p,\theta-p}(\cD)}
}
\nnrm{\varphi}{\wsob^{1}_{p',\Theta',\theta'}(\cD)},
\quad \varphi\in  \wsob^{1}_{p',\Theta',\theta'}(\cD),
\]
which means, in particular, that $v^*\in \wsob^{-1}_{p,\theta}(\cD)$.

\medskip

\noindent\textit{Step 3.} Due to the density of $\cont^\infty_0(\cD)$ in $ \wsob^{1}_{p',\Theta',\theta'}(\cD)$, $J$ is obviously injective. 

\medskip

\noindent\textit{Step 4.} We prove that
\begin{align}\label{eq:lem:dual:o-1:isometry}
\nnrm{v^*}{\wsob^{-1}_{p,\Theta,\theta}(\cD)}
&=
\inf\sggklam{\nnrm{f^0}{L_{p,\Theta,\theta}(\cD)}
+ 
\nnrm{f^1}{L_{p,\Theta-p,\theta-p}(\cD)} 
+ 
\nnrm{f^2}{L_{p,\Theta-p,\theta-p}(\cD)}}\\
&=
\min\sggklam{\nnrm{f^0}{L_{p,\Theta,\theta}(\cD)}
+ 
\nnrm{f^1}{L_{p,\Theta-p,\theta-p}(\cD)} 
+ 
\nnrm{f^2}{L_{p,\Theta-p,\theta-p}(\cD)}},
\end{align}
the infimum being taken over all $f^0\in L_{p,\Theta,\theta}(\cD)$, $f^1\in L_{p,\Theta-p,\theta-p}(\cD)$, and $f^2\in L_{p,\Theta-p,\theta-p}(\cD)$, such that $Jv^*=f^0-f^1_{x^1}-f^2_{x^2}$.
To see that~\eqref{eq:lem:dual:o-1:isometry} holds, we argue as follows:
The calculations in Step~2 already show that
\[
\nnrm{v^*}{\wsob^{-1}_{p,\Theta,\theta}(\cD)}
\leq
\nnrm{Jv^*}{\Lambda^{-1}_{p,\Theta,\theta}(\cD)},
\qquad
v^*\in \wsob^{-1}_{p,\Theta,\theta}(\cD).
\]
Moreover, the arguments in Step~1 show that for every $v^*\in \wsob^{-1}_{p,\Theta,\theta}(\cD)$, there exist $f^0\in L_{p,\Theta,\theta}(\cD)$, $f^1\in L_{p,\Theta-p,\theta-p}(\cD)$, and $f^2\in L_{p,\Theta-p,\theta-p}(\cD)$, such that
\[
\nnrm{v^*}{\wsob^{-1}_{p,\Theta,\theta}(\cD)}
=
\nnrm{f^0}{L_{p,\Theta,\theta}(\cD)}
+
\nnrm{f^1}{L_{p,\Theta-p,\theta-p}(\cD)}
+
\nnrm{f^2}{L_{p,\Theta-p,\theta-p}(\cD)}.
\]
The reason: The extension $\tilde{v}$ of $\hat{v}$ is norm preserving, $P$ is an isometry and $(V^1_{p',\Theta',\theta'}(\cD))^*$ is isometrically isomorphic to $L_{p,\Theta,\theta}(\cD)\oplus_p L_{p,\Theta-p,\theta-p}(\cD) \oplus_p L_{p,\Theta-p,\theta-p}(\cD)$. 
This shows that~\eqref{eq:lem:dual:o-1:isometry} holds and that the infimum is indeed a minimum.

\medskip

\noindent\textit{Step 5.} Conclusion: Due to Steps~2--4, $\nnrm{\cdot}{\Lambda^{-1}_{p,\Theta,\theta}(\cD)}$ is indeed a norm on $\Lambda^{-1}_{p,\Theta,\theta}(\cD)$, the infimum in the definition is a minimum and $J$ is an isometric isomorphism.
\end{proof}

\begin{lemma}\label{lem:dual:parabolic:H1}
For $1<p<\infty$ and $\Theta,\theta\in\bR$, consider
\[
\Lambda^{-1}_{p,\Theta,\theta}(\cD,T)
:=
\sggklam{
g\colon \Omega_T\to\mathscr{D}'(\cD)\,\colon\,
g=\sum_{\abs{\alpha}\leq 1}(-1)^{\abs{\alpha}} D^\alpha v_\alpha,\, v_\alpha\in \bL_{p,\Theta-\abs{\alpha}p,\theta-\abs{\alpha}p}(\cD,T),\abs{\alpha}\leq 1
},
\]
endowed with the norm
\begin{align*}
\nnrm{g}{\Lambda^{-1}_{p,\theta}(\cD,T)}
:=
\inf
\sggklam{\sum_{\abs{\alpha}\leq 1} \nnrm{v_\alpha}{\bL_{p,\Theta-\abs{\alpha}p,\theta-\abs{\alpha}p}(\cD,T)}
}
=
\min
\sggklam{\sum_{\abs{\alpha}\leq 1} \nnrm{v_\alpha}{\bL_{p,\Theta-\abs{\alpha}p,\theta-\abs{\alpha}p}(\cD,T)}
},
\end{align*}
where the infimum is taken over all possible tuples $(v_\alpha)_{\abs{\alpha}\leq 1}$ with $v_{\alpha}\in \bL_{p,\Theta-\abs{\alpha}p,\theta-\abs{\alpha}p}(\cD,T)$, $\abs{\alpha}\leq 1$, such that $g=\sum_{\abs{\alpha}\leq 1}(-1)^{\abs{\alpha}} D^\alpha v_\alpha$.
Then the mapping
\begin{align*}
J\colon \bwsob^{-1}_{p,\theta}(\cD,T) &\to \Lambda^{-1}_{p,\theta}(\cD,T)\\
f&\mapsto f|_{\cont_0^\infty(\cD)},
\end{align*}
is an isometric isomorphism. In this sense, we can say that
\begin{align*}
\mathbb K^{-1}_{p,\theta}(\cD,T)
=
\sggklam{f^0-f^1_{x^1}-f^2_{x^2}\,\colon\, f^0\in \bL_{p,\Theta,\theta}(\cD,T),\, f^i\in \bL_{p,\Theta-p,\theta-p}(\cD,T),\,i=1,2},
\end{align*}
and that
\[
\nnrm{f}{\bwsob^{-1}_{p,\Theta,\theta}(\cD,T)}
=
\min
\sggklam{\nnrm{f^0}{\bL_{p,\Theta,\theta}(\cD,T)}
+
\nnrm{f^1}{\bL_{p,\Theta-p,\theta-p}(\cD,T)}
+
\nnrm{f^2}{\bL_{p,\Theta-p,\theta-p}(\cD,T)}},
\]
the infimum being taken over all possible representations of $f|_{\cont^\infty_0(\cD)}$ as $f^0+f^1_{x^1}+f^2_{x^2}$ with $f^0\in \bL_{p,\Theta,\theta}(\cD,T)$ and $f^i\in \bL_{p,\Theta-p,\theta-p}(\cD,T)$, $i=1,2$.
%

\end{lemma}

\begin{proof}
This can be proven along the lines of the proof of Lemma~\ref{lem:dual:K1} by using Lemma~\ref{lem:dual:parabolic:Lo} instead of Lemma~\ref{lem:dual:Lo}.
The details are left to the reader.
\end{proof}

\end{appendix}

\providecommand{\bysame}{\leavevmode\hbox to3em{\hrulefill}\thinspace}
\providecommand{\MR}{\relax\ifhmode\unskip\space\fi MR }
\providecommand{\MRhref}[2]{%
  \href{http://www.ams.org/mathscinet-getitem?mr=#1}{#2}
}
\providecommand{\href}[2]{#2}

\end{document}